\documentclass[pdflatex,sn-mathphys-num]{sn-jnl}

\usepackage{graphicx}%
\usepackage{multirow,hhline}%
\usepackage{amsmath,amssymb,amsfonts}%
\usepackage{amsthm}%
\usepackage{mathrsfs}%
\usepackage[title]{appendix}%
\usepackage{xcolor}%
\usepackage{textcomp}%
\usepackage{manyfoot}%
\usepackage{algorithm}%
\usepackage{algorithmicx}%
\usepackage{algpseudocode}%
\usepackage{listings}%

\allowdisplaybreaks[4]  %

\newtheorem{theorem}{Theorem}[section]
\newtheorem{theorem1}{Theorem}[section]
\newtheorem{Assumption}[theorem1]{Assumption}
\newtheorem{theorem2}{Theorem}[section]
\newtheorem{Lemma}[theorem2]{Lemma}
\newtheorem{theorem3}{Theorem}[section]
\newtheorem{Corollary}[theorem3]{Corollary}
\newtheorem{theorem4}{Theorem}[section]
\newtheorem{Remark}[theorem4]{Remark}

\begin{document}

\title[Approximation of invariant measure for the SACE]{Approximation of the invariant measure for the stochastic Allen--Cahn equation via an explicit fully discrete scheme}


\author[1]{\fnm{Yibo} \sur{Wang}}\email{wangyb@seu.edu.cn}

\author*[1]{\fnm{Wanrong} \sur{Cao}}\email{wrcao@seu.edu.cn}

\affil[1]{\orgdiv{School of Mathematics}, \orgname{Southeast University}, \orgaddress{ \city{Nanjing}, \postcode{210096}, \country{China}}}

%
%
%
%
%
%


\abstract{We propose a novel explicit fully discrete scheme for the stochastic Allen--Cahn equation (SACE) driven by an additive \textit{Q}-Wiener process. The scheme combines the spectral Galerkin method for spatial discretization with a tamed accelerated exponential integrator for temporal approximation. It is designed to alleviate the computational cost of implicit methods, enabling efficient long-time simulations and explicit approximation of the invariant measure. Uniform moment bounds and uniform weak convergence of the numerical solution are established, both essential for approximating the invariant measure. A refined taming strategy yields time-independent weak error estimates, and the weak convergence rate is shown to be sharp, improving upon existing results. Furthermore, by combining weak convergence analysis with exponential ergodicity, we derive an error bound for the approximation of the invariant measure. Theoretical results confirm the efficiency and accuracy of the proposed method in capturing the long-term statistical behavior of the SACE.}




\maketitle

\section{Introduction}

The invariant measure plays a fundamental role in describing the long-term statistical behavior of stochastic evolution models, capturing the steady-state or equilibrium distribution of the underlying system. In recent years, significant progress has been made in the numerical approximation of invariant measures for nonlinear stochastic partial differential equations (SPDEs), particularly when the drift term satisfies a global Lipschitz condition \cite{Brehier2014, Brehier2016, Brehier2017, Chen2020}. These advances have deepened our understanding of long-term dynamics and provided valuable numerical tools for a wide range of scientific and engineering applications.

Among such models, the stochastic Allen--Cahn equation (SACE) serves as an effective phase field model for describing the evolution of interfaces under stochastic influences. The numerical analysis of the SACE, particularly for cases with additive noise, has been extensively studied. Prior works have established strong convergence results for fully discrete schemes, where the temporal discretization includes the backward Euler scheme \cite{QiWang2019}, operator splitting methods \cite{Brehier2019, BrehierGoude2019}, and adaptive time-stepping techniques \cite{ChenDangHong2024}. For the multiplicative noise-driven SACE, significant contributions in \cite{Feng2017, HuangShen2023,  XuCJ2023} have examined the impact of noise on numerical stability and convergence properties.

A key challenge in the numerical approximation of invariant measures is ensuring uniform moment bounds for numerical solutions that remain independent of time. This difficulty is particularly pronounced in nonlinear SPDEs, such as the SACE, where the drift term exhibits superlinear growth. Existing research has explored various approaches to address this issue. For instance, Br\'ehier et al. \cite{Brehier2016} introduced a postprocessed integrator to improve the convergence rate of numerical invariant measures. While this method enhances the accuracy of the standard linearized implicit Euler scheme for additive noise-driven SACEs, it remains restricted to linear SPDEs and does not fully address time-independent convergence analysis, which is crucial for long-time simulations. Similarly, Cui et al. \cite{Cui2021} investigated the invariant measure of parabolic SPDEs with non-globally Lipschitz coefficients, employing a spectral Galerkin method combined with an implicit Euler scheme. Their work established sharp weak convergence rates, demonstrating that the invariant measure could be accurately approximated.
For SACEs driven by multiplicative noise, Liu et al. \cite{Liu2024,Liu2023} analyzed drift-implicit Euler-based fully discrete schemes and proved the unique ergodicity of their numerical approximations. However, both \cite{Cui2021} and \cite{Liu2024,Liu2023} rely on implicit time discretization, which necessitates solving nonlinear equations iteratively at each time step. This significantly increases computational costs, particularly for long-time simulations, making explicit schemes a more desirable alternative.

To address these challenges, this paper proposes a novel explicit fully discrete scheme for the SACE driven by an additive $Q$-Wiener process in the Hilbert space $H := L^2(\mathcal{I})$, where $\mathcal{I} := (0,1)$ is the spatial domain. The proposed scheme is designed to circumvent the limitations of implicit methods, providing a computationally efficient alternative that is particularly suitable for long-time simulations and invariant measure approximation.
The model is given by: 
\begin{equation}\label{SACE}
	\left\{
	\begin{aligned}
		&du(t) + Au(t)dt = F(u(t))dt + dW(t), \quad t>0, \\
		&u(0) = u_0 \in H, 
	\end{aligned}
	\right.
\end{equation}
where $-A$ is the Laplacian operator with homogeneous Dirichlet boundary conditions, and $F$ is a Nemytskii operator of a real-valued one-sided Lipschitz function $f$, as specified in Assumption \ref{Asp:f}. The process $W(t)$ represents a (possibly cylindrical) $Q$-Wiener process on a stochastic basis $(\Omega, \mathcal{F}, \mathbb{P}, \{\mathcal{F}_t\}_{t \geq 0})$.  Under suitable conditions, the SACE possesses a unique mild solution, given by: 
\begin{equation}\label{mild solution1} 
	u(t) = S(t) u_0 + \int_0^t S(t-s) F(u(s))ds + \mathcal{O}_t, 
\end{equation} 
where $S(t) = e^{-tA}$ is the analytic semigroup generated by $-A$, and $\mathcal{O}_t := \int_0^t S(t-s) dW(s)$ denotes the stochastic convolution. Under a strong dissipation condition, the SACE is known to possess a unique ergodic invariant measure, as established in Theorem \ref{Th:ergodicity} and supported by previous studies \cite{Brehier2022,Cui2021}.

In this context, Br\'ehier \cite{Brehier2022} investigated the long-term behavior of an explicit tamed exponential Euler scheme, given by:
\begin{equation*}
	X_{n+1} = e^{-\Delta tA} X_{n} + A^{-1} (I-e^{-\Delta tA}) \frac{F(X_n)}{1+\Delta t \|F(X_n)\|_{L^2}} + e^{-\Delta tA} \Delta W_n.
\end{equation*}
This scheme allowed for moment estimates of the semi-discrete solutions, leading to the following bound:
\begin{equation*}
	\sup\limits_{0\leq n\Delta t \leq T} (\mathbb{E}[\|X_n\|_{L^{\infty}}^p])^{\frac{1}{p}} \leq (1+T) \mathcal{P}(\|x_0\|_{L^{\infty}}), \quad \forall T \in (0, \infty), 
\end{equation*}
where $X_n$ represented semi-discrete solutions, $\Delta t$ was the time step size and $\mathcal{P}$ was a polynomial function.  However, since this moment bound depends on the time horizon $T$, it does not provide a time-independent structure, thereby limiting its direct applicability to the numerical approximation of the invariant measure over long time horizons.

Building upon Br\'ehier's approach \cite{Brehier2022}, this work advances the state of the art by introducing a novel explicit fully discrete scheme and establishing a time-independent weak convergence analysis. This not only enables an efficient numerical approximation of the invariant measure but also provides a rigorous theoretical foundation for the long-term behavior of numerical solutions to the SACE \eqref{SACE}.

The primary contributions of this paper are as follows.

\vspace{0.5em}
$\bullet$ \textbf{Development of a novel explicit fully discrete scheme:}
We propose an explicit fully discrete numerical scheme for the SACE \eqref{SACE}:
\begin{equation*}
	u^{N}_{k+1} = S_N(\tau)u^{N}_{k}
	+ \frac{A_N^{-1}(I-S_N(\tau))F_N(u^{N}_{k})}{1 + \tau^{\beta}\|u^{N}_{k}\|_{L^\infty}^6 + \tau^{\beta}\|u^{N}_{k}\|_{\dot{H}^{\beta}}^6 }
	+ \int_{t_k}^{t_{k+1}} S_N(t_{k+1}-s) P_N dW(s).
\end{equation*} 
Unlike existing implicit schemes used in \cite{Cui2021,Liu2024,Liu2023}, our explicit approach eliminates the need for iterative solvers at each time step, significantly reducing computational complexity. A key theoretical result (Theorem \ref{Th:moment bound v}) establishes that the numerical solution satisfies uniform moment bounds independent of time, making it well-suited for long-term simulations. Moreover, leveraging this scheme, we provide an efficient and explicit method to approximate the invariant measure of the underlying stochastic system.

\vspace{0.5em}
$\bullet$ \textbf{Uniform moment bounds and sharp weak convergence analysis:}
We establish uniform moment bounds for the numerical solution, which play a crucial role in the subsequent weak convergence analysis. A fundamental challenge in weak error estimation is handling the superlinear growth of the nonlinear term, which we address by introducing a novel taming strategy. This approach enables a rigorous weak error analysis. Using Malliavin calculus, we derive the following time-independent weak error estimate:
\begin{theorem}\label{Th:full-discrete order}
	Let Assumptions  \ref{Asp:Initial value}--\ref{Asp:dissipative}, \ref{Asp:additional asp} hold, and assume $\|A^{\frac{\beta-1}{2}}Q^{\frac{1}{2}}\|_{\mathcal{L}_2(H)}<\infty$ with $\beta\in(0,1]$. Then for any test functions $\Phi\in C_b^2(H;\mathbb{R})$, $K\geq2$, and $\tau\in(0,\tau_0]$ with some arbitrary $\tau_0>0$, there exists a constant $C>0$ independent of $N$, $\tau$, $K$, and $T$, such that
	\begin{align*} 
		\left| \mathbb{E}[\Phi(u(K\tau))] - \mathbb{E}[\Phi(u^N_{K})] \right| \leq C (1+{(K\tau)}^{-\beta}) (\lambda_{N+1}^{-\beta}+\tau^{\beta}), 
	\end{align*} 
	where $u^N_K$ represent the fully discrete solutions given by \eqref{scheme1}. 
\end{theorem}
This result demonstrates that our scheme attains a sharp weak convergence rate, improving upon prior estimates, which often only achieve $\beta - \epsilon$ order weak convergence for arbitrarily small $\epsilon$ \cite{Brehier2022,Cai2021,Cui2021}.

\vspace{0.5em}
$\bullet$ \textbf{Application to the numerical approximation of invariant measures:}
Beyond their role in practical applications (e.g., option pricing in finance \cite{Bally1996,Bossy2021,WangZhaoZhang2024}), weak error estimates are essential for analyzing numerical approximations of invariant measures. The approximation error of the invariant measure ultimately reduces to the weak error, which typically enjoys a higher convergence rate than the strong error. Combining the weak error analysis with the exponential ergodicity of the continuous system (Theorem \ref{Th:ergodicity}), we rigorously derive the numerical error in approximating the invariant measure:
\begin{Corollary}\label{Corollary:1} 
	Under the same conditions as Theorem \ref{Th:full-discrete order}, for any test functions $\Phi\in C_b^2(H;\mathbb{R})$, $\beta\in(0,1]$, $\tau\in(0,\tau_0]$, and sufficiently large $K$, it holds that 
	\begin{align*} \left| \mathbb{E}[\Phi(u^N_K)] - \int_{H} \Phi d\mu \right| \leq C(u_0,Q,\Phi) (e^{-(\lambda_{1}-L_F)K\tau} + \lambda_{N+1}^{-\beta} + \tau^{\beta}), \end{align*} where $\mu$ denotes the invariant measure of the SACE \eqref{SACE}.
\end{Corollary}
This result provides a theoretical guarantee for the accuracy of numerical invariant measure approximations, reinforcing the practical significance of our approach.

The remainder of this paper is structured as follows. Section \ref{Sec:Preliminary} introduces essential notations, assumptions, and preliminary results, including the regularity and ergodicity of the mild solution and relevant properties from Malliavin calculus. In Section \ref{Sec:Spatial discrete}, we analyze the weak error of the spatial semi-discretization. Section \ref{Sec:Full discrete} presents the fully discrete scheme, establishes the uniform moment bounds, and proves Theorem \ref{Th:full-discrete order}. 
In Section \ref{Sec:numerical example}, we perform the numerical experiment to validate the theoretical results. 
Finally, Section \ref{section:conclusion} concludes the paper with a summary and discussion of potential future research directions.

\section{Preliminary}\label{Sec:Preliminary}
Recalling that $\mathcal{I} = (0,1)$, we denote by $L^p(\mathcal{I})$, for $p \geq 1$, the Banach space of real-valued, $p$-integrable functions, endowed with the associated norm $\|\cdot\|_{L^p(\mathcal{I})}$ (abbreviated as $\|\cdot\|_{L^p}$). 
When $p=2$, we simply write $H = L^2(\mathcal{I})$ with the inner product $\langle \cdot , \cdot \rangle$. 
Given a Banach space $(\mathcal{V}, \|\cdot\|_{\mathcal{V}})$, we denote by $L^p(\Omega; \mathcal{V})$ the space of $\mathcal{V}$-valued, $p$-integrable random variables, endowed with the norm 
$\|v\|_{L^p(\Omega; \mathcal{V})} = (\mathbb{E} \|v\|_{\mathcal{V}}^p)^{\frac{1}{p}}$. 
For a separable Hilbert space $(\mathcal{U}, \langle\cdot,\cdot\rangle_{\mathcal{U}}, \|\cdot\|_{\mathcal{U}})$, we denote by $\mathcal{L}(\mathcal{U})$ the space of all bounded linear operators on $\mathcal{U}$, endowed with the usual operator norm $\|\cdot\|_{\mathcal{L}(\mathcal{U})}$. 
By $\mathcal{L}_2(\mathcal{U})$, we denote the Hilbert space of Hilbert--Schmidt operators on $\mathcal{U}$, with the inner product
$\langle T_1, T_2 \rangle_{\mathcal{L}_2(\mathcal{U})} := \sum_{k\in\mathbb{N}} \langle T_1\eta_{k}, T_2\eta_{k}\rangle_{\mathcal{U}}$, 
and the corresponding Hilbert--Schmidt norm $\|T\|_{\mathcal{L}_2(\mathcal{U})} := \langle T,T \rangle_{\mathcal{L}_2(\mathcal{U})}^{\frac{1}{2}},$
where $\mathbb{N}$ represents the set of natural numbers, and $\{\eta_{k}\}_{k\in\mathbb{N}}$ is an orthonormal basis of $\mathcal{U}$. 
By $C_b^2(H;\mathbb{R})$, we denote the space of twice continuously differentiable functionals from $H$ to $\mathbb{R}$, which are not necessarily bounded, but whose Fr\'echet derivatives are bounded up to order 2.

Throughout the paper, we use $C$ to denote a generic positive constant, which may vary from one line to another but is always independent of the discretization parameters. 
We sometimes write $C(a,b,...)$ to specify its dependence on certain parameters $a,b,...$. 

We denote by $H^k(\mathcal{I})$ and $H^k_0(\mathcal{I})$, for $k \in \mathbb{N}$, the usual Sobolev spaces. 
There exists an orthonormal basis $\{\phi_k\}_{k\in\mathbb{N}}$ of $H$, and an increasing sequence $\{\lambda_{k}\}_{k\in\mathbb{N}}$ such that $A\phi_k = \lambda_{k}\phi_k$. 
The fractional powers of $A$, i.e., the operators $A^{\frac{s}{2}}$, are defined in the same manner as in \cite[Appendix B.2]{Kruse2014-1}. 
We introduce a family of separable Hilbert spaces $\dot{H}^s$, $s \in \mathbb{R}$, by setting 
$\dot{H}^s := \mathrm{dom}(A^{\frac{s}{2}})$, equipped with the inner product 
$\langle u, v \rangle_{\dot{H}^s} := \langle A^{\frac{s}{2}}u, A^{\frac{s}{2}}v \rangle$ and the norm 
$\|u\|_{\dot{H}^s} = \langle u, u \rangle_{\dot{H}^s}^{\frac{1}{2}}$. 

The properties of the semigroup $S(t)$ are stated as follows, which can be found in, e.g., \cite{Chen2020, Chen2023}, \cite[Lemma 2.5]{Kruse2014-2}. 
\begin{align}
	\|A^\nu S(t)\|_{\mathcal{L}(H)} &\leq C(\nu) t^{-\nu} e^{-\frac{\lambda_{1}}{2}t}, \quad t > 0, \ \nu \geq 0, \label{semigroup1} \\
	\|A^{-\rho}(S(t)-S(s))\|_{\mathcal{L}(H)} &\leq C (t-s)^{\rho} e^{-\frac{\lambda_{1}}{2}s}, \quad 0 \leq s \leq t, \ \rho \in [0,1], \label{semigroup2} \\
	\Big\| A^{\rho} \int_{s}^{t} S(t-r) \psi \, dr \Big\| &\leq C(t-s)^{1-\rho} \|\psi\|, \quad \psi \in H, \ 0 \leq s < t, \ \rho \in [0,1].  \label{semigroup_sharp}
\end{align} 

We assume that the covariance operator $Q$ of the $Q$-Wiener process $W(t)$ is self-adjoint and non-negative definite, commutes with $A$, and satisfies
\begin{equation*}
	\|A^{\frac{\beta-1}{2}}\|_{\mathcal{L}_2^0}
	= \|A^{\frac{\beta-1}{2}}Q^{\frac{1}{2}}\|_{\mathcal{L}_2(H)} < \infty, \quad \beta \in (0,1], 
\end{equation*} 
where $\mathcal{L}_2^0 := \mathcal{L}_2(U_0; H)$ denotes the space of Hilbert--Schmidt operators from the Hilbert space $U_0 := Q^{1/2}(H)$ to $H$. 
Note that two important types of noise, space-time white noise for $\beta = \frac{1}{2} - \epsilon$ (i.e., $Q = I$), and trace-class noise for $\beta = 1$, are included in this setting. 

\begin{Assumption}\label{Asp:Initial value}
	Let $u_0: \Omega \to H$ be an $\mathcal{F}_0$-measurable random variable. 
	Assume that for some sufficiently large positive integer $\tilde{p}$, 
	\begin{equation*}
		\|u_0\|_{L^{\tilde{p}}(\Omega; \dot{H}^{\beta})} + \|u_0\|_{L^{\tilde{p}}(\Omega; L^{\infty})} < \infty. 
	\end{equation*}
\end{Assumption}

\begin{Assumption}\label{Asp:f}
	Let $f(x) := -a_3 x^3 + a_2 x^2 + a_1 x + a_0$ be a real-valued polynomial with $a_3 > 0$, $a_i \in \mathbb{R}$, $i = 0,1,2$, and let $F: L^6(\mathcal{I}) \to H$ be the Nemytskii operator given by 
	\[
	F(v)(x) := f(v(x)) = -a_3 v^3 + a_2 v^2 + a_1 v + a_0. 
	\]
\end{Assumption}
For $v, \zeta, \zeta_1, \zeta_2 \in L^6(\mathcal{I})$, it holds that
\begin{align*}
	(F'(v)(\zeta))(x) &= f'(v(x)) \zeta(x) = ( -3a_3 v^2(x) + 2a_2 v(x) + a_1 ) \zeta(x), \\
	(F''(v)(\zeta_1, \zeta_2))(x) &= f''(v(x)) \zeta_1(x) \zeta_2(x) = ( -6a_3 v(x) + 2a_2 ) \zeta_1(x) \zeta_2(x).
\end{align*}
By denoting $L_F := \sup_{x \in \mathbb{R}} f'(x)=  a_1+\frac{a_2^2}{3a_3} < +\infty$, it is straightforward to obtain the following estimates:
\begin{equation}\label{Est:F1}
	\begin{aligned}
		\langle u-v, F(u)-F(v) \rangle &\leq L_F \|u-v\|^2, \\
		\langle u, F'(v)u \rangle &\leq L_F \|u\|^2.
	\end{aligned}
\end{equation}
Moreover, there exists a constant $C_F > 0$ such that
\begin{equation}\label{Est:F2}
	\begin{aligned}
		\|F(u)-F(v)\| &\leq C_F ( \|u\|_{L^\infty}^2 + \|v\|_{L^\infty}^2 + 1 ) \|u-v\|, \\
		\|F'(v)u\| &\leq C_F ( \|v\|_{L^\infty}^2 + 1 ) \|u\|, \\
		\|F''(\zeta)(u,v)\|_{\dot{H}^{-1}} &\leq C_F ( \|\zeta \|_{L^\infty} + 1 ) \|u\| \|v\|.
	\end{aligned}
\end{equation}
The regularity properties of the stochastic convolution have been extensively studied (see, e.g., \cite{Brehier2019,Wang2020}). Some of these are summarized in the following lemma.

\begin{Lemma}\label{Le:regularity O}
	For $p \geq 2$, $\beta \in (0, 1]$ and $\gamma \in [0, \beta]$, it holds that
	\begin{align*}
		&	\sup_{t \geq 0} \|\mathcal{O}_t\|_{L^p(\Omega;\dot{H}^\beta)} + \sup_{t \geq 0} \|\mathcal{O}_t\|_{L^p(\Omega;L^\infty)} < \infty, \\
		&	\|\mathcal{O}_t-\mathcal{O}_s\|_{L^p(\Omega;\dot{H}^\gamma)} \leq C (t-s)^{\frac{\beta-\gamma}{2}}, \quad 0 \leq s \leq t.
	\end{align*}
\end{Lemma}

The regularity properties of $u(t)$ are summarized in the following theorem. We omit the proof since it is standard and refer to \cite[Chapter 6]{Cerrai2001}, \cite[Proposition 3.1]{Brehier2022} for further details.

\begin{theorem}
	Under Assumptions \ref{Asp:Initial value}--\ref{Asp:f}, for $p \geq 2$ and $\beta \in (0, 1]$, it holds that
	\begin{equation}\label{regularity u}
		\sup_{t \geq 0} \| u(t) \|_{L^p(\Omega; L^\infty)} + \sup_{t \geq 0} \| u(t) \|_{L^p(\Omega; \dot{H}^\beta)} \leq C(u_0, p, Q) < \infty.
	\end{equation}
	Furthermore, for $0 \leq s < t$,
	\begin{equation}\label{Holder u}
		\|u(t)-u(s)\|_{L^p(\Omega; H)} \leq C(u_0, p, Q) (t - s)^{\frac{\beta}{2}}.
	\end{equation}
\end{theorem}

If necessary, we write $u(t; u_0)$ to emphasize the dependence on the initial value $u_0$ of the mild solution to \eqref{SACE}. In order to derive the ergodicity of the SACE and the weak convergence error of numerical approximations over an infinite time interval, the following assumption is required.

\begin{Assumption}\label{Asp:dissipative}
	Assume that $L_F < \lambda_1$.
\end{Assumption}

\begin{theorem}[Proposition 3.3 in \cite{Brehier2022}]\label{Th:ergodicity}
	Under Assumptions \ref{Asp:Initial value}--\ref{Asp:dissipative}, there exists a unique invariant measure $\mu$ for the SACE \eqref{SACE}. Moreover, for a Lipschitz continuous function $\varphi: L^p(\mathcal{I}) \to \mathbb{R}$ and for any $z \in L^p(\mathcal{I})$, $p \geq 2$,
	\begin{equation}\label{exponential decay}
		\Big| \mathbb{E}[\varphi(u(t; z))] - \int_{L^p(\mathcal{I})} \varphi \, d\mu \Big| \leq C(\varphi) e^{-(\lambda_1 - L_F)t} (\| z \|_{L^p} + 1).
	\end{equation}
\end{theorem}

\begin{Remark}\label{remark}
	From Theorem \ref{Th:ergodicity}, it follows that for an $L^p$-valued random variable $u_0: \Omega \rightarrow L^p$, we have
	\begin{align*}
		\Big| \mathbb{E}[\varphi(u(t;u_0))] - \int_{L^p(\mathcal{I})} \varphi \, d\mu \Big|
		\leq C(\varphi) e^{-(\lambda_1 - L_F)t} ( \mathbb{E}[\|u_0\|_{L^p}] + 1 ),
	\end{align*}
	where the constant $C(\varphi)$ depends on the test function $\varphi$.
\end{Remark}

Next, we introduce some notations related to Malliavin calculus that will be used in the subsequent analysis. For a comprehensive treatment, we refer to the monograph by Nualart \cite{Nualart2006}.

Let $\mathcal{W}: L^2([0,T]; U_0) \rightarrow L^2(\Omega)$ be an isonormal process, where for any $\psi \in L^2([0,T]; U_0)$, $\mathcal{W}(\psi)$ is a centered Gaussian random variable, and the covariance satisfies
\[
\mathbb{E}[\mathcal{W}(\psi_1) \mathcal{W}(\psi_2)] = \langle \psi_1, \psi_2 \rangle_{L^2([0,T]; U_0)},
\quad \psi_1, \psi_2 \in L^2([0,T]; U_0).
\]
For given $\psi_m \in L^2([0,T]; U_0)$ with $m = 1, 2, \dots, M$ and $h_k \in H$ with $k=1,2,\dots, K$, where $K, M \in \mathbb{N}$, we define a family of smooth $H$-valued random variables as follows:
\begin{equation*}
	\mathcal{S}(H) := \Big\{ G = \sum_{k=1}^{K} f_k( \mathcal{W}(\psi_1), \mathcal{W}(\psi_2), \dots, \mathcal{W}(\psi_M) ) h_k : f_k \in C_p^\infty(\mathbb{R}^M; \mathbb{R}) \Big\},
\end{equation*}
where $C_p^\infty(\mathbb{R}^M; \mathbb{R})$ represents the space of smooth functions from $\mathbb{R}^M$ to $\mathbb{R}$ with polynomial growth in their derivatives. 

The Malliavin derivative of $G \in \mathcal{S}(H)$ is defined as
\begin{equation*}
	\mathcal{D}_t G := \sum_{k=1}^{K} \sum_{m=1}^{M} \partial_m f_k ( \mathcal{W}(\psi_1), \mathcal{W}(\psi_2), \dots, \mathcal{W}(\psi_M) ) h_k \otimes \psi_m(t),
\end{equation*}
where the tensor product $h_k \otimes \psi_m(t) \in \mathcal{L}_2^0$ for $k = 1, 2, \dots, K$ and $m = 1, 2, \dots, M$ is defined by
\begin{equation*}
	(h_k \otimes \psi_m(t))(\varphi) := \langle \psi_m(t), \varphi \rangle_{U_0} h_k \quad \text{for all } \varphi \in U_0, \ h_k \in H, \ t \in [0,T].
\end{equation*}
We denote by $\mathcal{D}_t^{\varphi} G = \mathcal{D}_t G (\varphi)$ the Malliavin derivative of $G$ in the direction $\varphi \in U_0$. 

We now list some fundamental properties of the Malliavin derivative. Firstly, if $G$ is $\mathcal{F}_t$-measurable and $s > t$, then $\mathcal{D}_s G = 0$. Additionally, since $\mathcal{D}_t$ is a closable operator, we define the space $\mathbb{D}^{1,2}(H)$ as the closure of $\mathcal{S}(H)$ under the following norm
\begin{equation*}
	\|G\|_{\mathbb{D}^{1,2}(H)} = \left( \mathbb{E}[\|G\|^2] + \int_0^T \mathbb{E}[\|\mathcal{D}_s G\|_{\mathcal{L}_2^0}^2] \, ds \right)^{\frac{1}{2}}.
\end{equation*}
For a separable Hilbert space $\mathcal{U}$, the chain rule holds:
\begin{equation}\label{chain rule}
	\mathcal{D}_t^{\varphi}(g(G)) = g'(G) \cdot \mathcal{D}_t^{\varphi} G, \quad g \in C_b^1(H;\mathcal{U}), \ G \in \mathbb{D}^{1,2}(H), \ \varphi \in U_0. 
\end{equation}
Finally, we quote the Malliavin integration by parts formula. For $G \in \mathbb{D}^{1,2}(H)$ and an adapted process $\Theta \in L^2([0,T]; \mathcal{L}_2^0)$, we have
\begin{equation}\label{Malliavin integral by parts}
	\mathbb{E}\left[ \left\langle \int_{0}^{T} \Theta(t)dW(t),G \right\rangle \right] 
	= \mathbb{E}\left[ \int_{0}^{T} \left\langle \Theta(t),\mathcal{D}_tG \right\rangle_{\mathcal{L}_2^0} dt \right]. 
\end{equation}
The Malliavin derivative of the It\^o integral $\int_0^t \Theta(r) \, dW(r)$ satisfies
\begin{equation}\label{Malliavin Ito}
	\mathcal{D}_s^{\varphi} \int_{0}^{t} \Theta(r) dW(r) = \int_{0}^{t} \mathcal{D}_s^{\varphi} \Theta(r) dW(r) + \Theta(s) \varphi, \quad  \varphi \in U_0, \ 0\leq s \leq t \leq T. 
\end{equation}

\section{Error estimate for the spatial semi-discretization}\label{Sec:Spatial discrete}
In this section, we employ the spectral Galerkin method for spatial semi-discretization and derive the weak approximation error. 
For $N \in \mathbb{N}$, we define the spectral Galerkin projection space $H^N := \text{span} \{ \phi_1,\phi_2,...,\phi_N \} \subset H$,  and the projection $P_N: \dot{H}^s \rightarrow H^N$ by
\begin{equation*}
	P_N \psi = \sum_{k=1}^{N} \langle \psi, \phi_k \rangle \phi_k, \quad \psi \in \dot{H}^s, \, s \in \mathbb{R}.
\end{equation*}
It follows immediately that
\begin{equation}\label{projection}
	\|(P_N-I)\psi\| \leq \lambda_{N+1}^{-\frac{s}{2}} \|\psi\|_{\dot{H}^s}, \quad \psi \in \dot{H}^s , \ s \geq 0. 
\end{equation}

The spectral Galerkin method applied to \eqref{SACE} consists of constructing $u^{N} \in H^N$, which satisfies 
\begin{equation*}
	du^{N}(t) + A_N u^{N}(t) dt
	= F_N(u^{N}(t)) dt + P_N dW, \quad 
	u^N(0) = P_N u_0,
\end{equation*}
where $A_N := P_N A$ and $F_N := P_N F$. 
The mild solution of the semi-discretization can be written as
\begin{equation}\label{mild uN}
	u^{N}(t) = S_{N}(t) P_N u_0 + \int_{0}^{t} S_{N}(t-s) F_N(u^{N}(s)) ds + \mathcal{O}_t^N,
\end{equation}
where $S_N(t) := e^{-t A_N}$, $t \geq 0$, is the analytic semigroup generated by $-A_N$, and $\mathcal{O}_t^N := \int_{0}^{t} S_N(t-s) P_N \, dW(s)$.

The next lemma provides smoothness properties of the semigroup $S_N(t)$, $t > 0$. Its proof follows an argument analogous to that in \cite[Lemmas 3.2, 4.1]{Wang2020} and is therefore omitted. 
\begin{Lemma} 
	For any $N\in\mathbb{N}$ and $t>0$, it holds that 
	\begin{align}
		\|P_NS(t)\psi\|_{L^p} &\leq C t^{-\frac{p-2}{4p}} e^{-\frac{\lambda_{1}t}{2}} \|\psi\|, \quad p\geq2, \label{Est:PN S2} \\ 
		\|P_NS(t)\psi\|_{L^\infty} &\leq C t^{-\frac{1}{2}} e^{-\frac{\lambda_{1}t}{4}} \|\psi\|_{L^1},  \label{Est:PN S3} \\ 
		\|P_NS(t)\psi\|_{L^3} &\leq C t^{-\frac{5}{24}} e^{-\frac{\lambda_{1}t}{4}} \|\psi\|_{L^\frac{4}{3}},  \label{Est:PN S4} \\
		\|P_N S(t) \psi\|_{L^{\infty}} &\leq C t^{\frac{2\nu - 1}{4}} e^{-\frac{\lambda_1}{2} t} \|\psi\|_{\dot{H}^{\nu}}, \quad  \nu \in [0, \tfrac{1}{2}). \label{Est:PN S 1} 
	\end{align}
\end{Lemma}

The boundedness of moments of $u^N(t)$ relies on the estimate of a perturbed differential equation on $H^N$, 
\begin{equation*}
	\frac{dw}{dt}=-A_N w + F_N(w+z), \quad w(0)=0, 
\end{equation*}
which possesses a unique mild solution, given by 
\begin{equation}\label{perturbed 1}
	w(t) = \int_{0}^{t} S_N(t-s) F_N(w(s)+z(s)) ds.  
\end{equation}
For brevity, for $p\geq2$, $t\geq0$ and $\varrho\geq0$, we introduce the following notation, 
\begin{equation*}
	\|z\|_{\mathbb{L}_{\varrho}^p(\mathcal{I}\times [0,t])} := \left( \int_{0}^{t} e^{-\varrho(t-s)} \|z(s)\|_{L^p}^p ds \right)^{\frac{1}{p}}.
\end{equation*} 
Sometimes we also write $\|z\|_{\mathbb{L}_{\varrho}^p}$ for short. 
The following lemma gives the upper bound of $L^\infty$-norm of $w(t)$, which is proved in Appendix \ref{Appendix perturbed}. 
\begin{Lemma}\label{lem:lem3.2}
	Under Assumption \ref{Asp:f}, for $w,z \in H^N$ given by \eqref{perturbed 1}, there exists a constant $C>0$ independent of $N$ and $t$, such that 
	\begin{equation}\label{Est:w L8}
		\|w(t)\|_{L^\infty} \leq C \big( \|z\|^9_{\mathbb{L}_{\lambda_{1}/8}^9(\mathcal{I}\times[0,t])} + 1 \big), \quad \forall t\geq0. 
	\end{equation}
\end{Lemma}

\begin{Assumption}\label{Asp:additional asp}
	Assume that $\|P_Nu_0\|_{L^{\tilde{p}}(\Omega;L^\infty)} + \|P_Nu_0\|_{L^{\tilde{p}}(\Omega;\dot{H}^1)} < \infty$ for sufficiently large $\tilde{p}\in\mathbb{N}$. 
\end{Assumption}

In the next lemma we give the spatial and temporal regularities of semi-discrete solution $u^N(t)$. 
\begin{Lemma}
	Under Assumptions \ref{Asp:Initial value}--\ref{Asp:f}, \ref{Asp:additional asp}, for $p\geq2$ and $\beta\in(0,1]$, it holds that 
	\begin{align}\label{regularity uN}
		\sup\limits_{N\in\mathbb{N},t\geq0} \|P_Nu(t)\|_{L^p(\Omega;L^\infty)} 
		+ \sup\limits_{N\in\mathbb{N},t\geq0} \|u^N(t)\|_{L^p(\Omega;L^\infty)}
		+ \sup\limits_{N\in\mathbb{N},t\geq0} \|u^N(t)\|_{L^p(\Omega;\dot{H}^{\beta})}
		< \infty. 
	\end{align} 
	Furthermore, for $0\leq s \leq t$, it holds that  $\|u^N(t)-u^N(s)\|_{L^p(\Omega;H)} \leq C(t-s)^{\frac{\beta}{2}}$, and 
	\begin{equation}\label{Holder F(uN)}
		\|F(u^N(t))-F(u^N(s))\|_{L^p(\Omega;H)} 
		\leq C (t-s)^{\frac{\beta}{2}}. 
	\end{equation}
\end{Lemma}
\begin{proof}
	By denoting $\tilde{u}^{N}(t) := u^{N}(t) - (S_{N}(t)P_Nu_0 + \mathcal{O}_t^N)$, we can rewrite \eqref{mild uN} as 
	\begin{equation*} 
		\tilde{u}^{N}(t) = \int_{0}^{t} S_{N}(t-s) F_N\left( \tilde{u}^{N}(s) + S_{N}(s)P_Nu_0 + \mathcal{O}_s^N \right) ds, 
	\end{equation*}
	which, together with \eqref{Est:w L8}, means that 
	\begin{equation*} 
		\|\tilde{u}^{N}(t)\|_{L^\infty} \leq C \left( \|S_N(t)P_Nu_0\|_{\mathbb{L}^{9}_{\lambda_{1}/8}(\mathcal{I}\times [0,t])}^{9} 
		+ \|\mathcal{O}_t^N\|_{\mathbb{L}^{9}_{\lambda_{1}/8}(\mathcal{I}\times [0,t])}^{9} + 1 \right). 
	\end{equation*}
	It follows from Minkowski's inequality, the definition of $\|\cdot\|_{\mathbb{L}_{\varrho}^p(\mathcal{I}\times [0,t])}$, H\"older's inequality, and $\|S(t)u\|_{L^p}\leq C\|u\|_{L^p}$ that 
	\begin{align*} 
		&\|u^N(t)\|_{L^p(\Omega;L^\infty)} \leq C \left( 1 + \left(\mathbb{E}\big[\|S_{N}(t)P_Nu_0\|_{\mathbb{L}^{9}_{\lambda_{1}/8} (\mathcal{I}\times[0,t])}^{9p} \big]\right)^{\frac{1}{p}}
		+ \left(\mathbb{E}\big[\|\mathcal{O}_t^N\|_{\mathbb{L}^{9}_{\lambda_{1}/8} (\mathcal{I}\times[0,t])}^{9p}\big]\right)^{\frac{1}{p}} \right) \\
		&= C \left( 1 + \left(\mathbb{E}\left[\left(\int_{0}^{t} e^{-\frac{\lambda_{1}(t-s)}{8}} \|S_N(s)P_Nu_0\|_{L^9}^{9}ds\right)^{p}\right]\right)^{\frac{1}{p}} + \left(\mathbb{E}\left[\left(\int_{0}^{t} e^{-\frac{\lambda_{1}(t-s)}{8}} \|\mathcal{O}_s^N\|_{L^9}^{9}ds\right)^{p}\right]\right)^{\frac{1}{p}} \right) \\
		&\leq C + C\left(\mathbb{E}\left[ \int_{0}^{t} e^{-\frac{\lambda_{1}(t-s)}{8}} \|P_Nu_0\|_{L^9}^{9p}ds \right] \left(\int_{0}^{t}e^{-\frac{\lambda_{1}(t-s)}{8}} ds\right)^{p-1}\right)^{\frac{1}{p}} \\
		&\quad + C \left(\mathbb{E}\left[ \int_{0}^{t} e^{-\frac{\lambda_{1}(t-s)}{8}} \|\mathcal{O}_s^N\|_{L^9}^{9p}ds \right] \left(\int_{0}^{t}e^{-\frac{\lambda_{1}(t-s)}{8}} ds\right)^{p-1}\right)^{\frac{1}{p}}  \\
		&\leq C \left( 1 + \left(\int_{0}^{t} e^{-\frac{\lambda_{1}(t-s)}{8}} \mathbb{E}[\|P_Nu_0\|_{L^9}^{9p}]ds\right)^{\frac{1}{p}} + \left(\int_{0}^{t} e^{-\frac{\lambda_{1}(t-s)}{8}} \mathbb{E}[\|\mathcal{O}_s^N\|_{L^9}^{9p}]ds\right)^{\frac{1}{p}} \right) \\
		&\leq C \left( 1 + \|P_Nu_0\|_{L^{9p}(\Omega;L^9)}^9 + \sup_{t\geq0}\|\mathcal{O}_t^N\|_{L^{9p}(\Omega;L^9)}^9 \right). 
	\end{align*} 
	Consequently, 
	\begin{align*}
		&\|u^N(t)\|_{L^p(\Omega;L^\infty)}
		\leq \|\tilde{u}^{N}(t)\|_{L^p(\Omega;L^\infty)} + \|S_{N}(t)P_Nu_0\|_{L^p(\Omega;L^\infty)} + \|\mathcal{O}_t^N\|_{L^p(\Omega;L^\infty)} \\
		&\leq C \left( 1 + \|P_Nu_0\|_{L^{9p}(\Omega;L^9)}^9 + \sup_{t\geq0}\|\mathcal{O}_t^N\|_{L^{9p}(\Omega;L^9)}^9 \right)
		+ \|P_Nu_0\|_{L^p(\Omega;L^\infty)} + \|\mathcal{O}_t^N\|_{L^p(\Omega;L^\infty)}. 
	\end{align*}
	In view of the following regularity properties of $\mathcal{O}_t^N$ (see, e.g., \cite[Lemma 2]{Cui2021}), 
	\begin{equation}\label{regularity ON}
		\sup\limits_{N\in\mathbb{N}, t\geq0} \|\mathcal{O}_t^N\|_{L^p(\Omega;\dot{H}^\beta)} 
		+ \sup\limits_{N\in\mathbb{N}, t\geq0} \|\mathcal{O}_t^N\|_{L^p(\Omega;L^\infty)} < \infty,  \quad \forall p\geq2, 
	\end{equation}
	we get $\sup_{N\in\mathbb{N},t\geq0}\|u^N(t)\|_{L^p(\Omega;L^\infty)}<\infty$. 
	The remaining proof of the lemma is standard, so we omit it. 
\end{proof}

Next, we examine the moment bounds of the Malliavin derivative of $u^N(t)$. 
\begin{Lemma}\label{Le:Malliavin}
	Let Assumptions \ref{Asp:Initial value}--\ref{Asp:dissipative}, and \ref{Asp:additional asp} hold. Then the Malliavin derivative of $u^N(t)$ satisfies 
	\begin{equation*}
		\mathbb{E} \big[\|\mathcal{D}_su^N(t)\|_{\mathcal{L}_2^0}^2\big] 
		\leq C \left( (t-s)^{\beta-1} + 1 \right), \quad 0 \leq s < t. 
	\end{equation*}
\end{Lemma}
\begin{proof}
	By differentiating \eqref{mild uN} along the direction $\psi\in U_0$, then using \eqref{chain rule} and \eqref{Malliavin Ito}, we get 
	\begin{equation}\label{Malliavin of uN}
		\mathcal{D}_s^{\psi} u^N(t) = \int_{s}^{t} S_{N}(t-r) P_N F'(u^{N}(r)) \mathcal{D}_s^{\psi} u^N(r) dr + S_{N}(t-s) P_N \psi, \quad 0 \leq s \leq t. 
	\end{equation}
	We introduce $\Upsilon_s^N(t,\psi)$ to denote 
	\begin{equation}\label{Upsilon}
		\Upsilon_s^N(t,\psi) := \mathcal{D}_s^{\psi} u^N(t) - S_{N}(t-s) P_N \psi 
		= \int_{s}^{t} S_{N}(t-r) P_N F'(u^{N}(r)) \mathcal{D}_s^{\psi} u^N(r) dr, 
	\end{equation}
	which solves 
	\begin{align*}
		\frac{d}{dt} \Upsilon_s^N(t,\psi) = \left( -A_N + P_N F'(u^{N}(t)) \right) \Upsilon_s^N(t,\psi) + P_N F'(u^{N}(t)) S_{N}(t-s) P_N \psi, \quad
		\Upsilon_s^N(s,\psi)= 0. 
	\end{align*}
	By introducing an evolution operator $\Psi(t,r)$  with respect to the following linear equation
	\begin{equation*}
		\frac{d}{dt} \Psi(t,r)z = -A_N \Psi(t,r)z + P_N F'(u^{N}(t)) \Psi(t,r)z, \quad \Psi(r,r)z=z, 
	\end{equation*}
	$\Upsilon_s^N(t,\psi)$ can be formulated as
	\begin{equation}\label{Est:Upsilon1}
		\Upsilon_s^N(t,\psi) = \int_{s}^{t} \Psi(t,r) P_N F'(u^{N}(r)) S_{N}(r-s) P_N \psi dr. 
	\end{equation}
	Next we estimate $\Psi(t,r)z$. Through the integration by parts, \eqref{Est:F1} and Poincar\'e's inequality, we obtain 
	\begin{align*}
		\frac{1}{2} \frac{d}{dt} \|\Psi(t,r)z\|^2 &= \langle -A_N \Psi(t,r)z , \Psi(t,r)z \rangle + \langle F'(u^{N}(t)) \Psi(t,r)z , \Psi(t,r)z \rangle \\
		&\leq -\|\Psi(t,r)z\|_{\dot{H}^1}^2 + L_F \|\Psi(t,r)z\|^2 
		\leq (L_F-\lambda_{1}) \|\Psi(t,r)z\|^2.  
	\end{align*}
	Under Assumption \ref{Asp:dissipative}, i.e., $L_F<\lambda_{1}$, it follows from Gronwall's inequality that 
	\begin{equation*}
		\|\Psi(t,r)z\| \leq C\|z\|, 
	\end{equation*}
	where the constant $C$ is independent of $t$ and $r$. 
	This, combined with \eqref{Est:Upsilon1}, \eqref{Est:F2}, \eqref{semigroup1} and H\"older's inequality, yields that for $\alpha<\frac{1}{2}$ and $s<t$, 
	\begin{align*}
		&\|\Upsilon_s^N(t,\psi)\| 
		\leq \int_{s}^{t} \left\| \Psi(t,r) P_N F'(u^{N}(r))  S_{N}(r-s) P_N \psi \right\| dr \\ 
		&\leq C \int_{s}^{t} \left\| P_N F'(u^{N}(r))  S_{N}(r-s) P_N \psi \right\| dr \\
		&\leq C C_F \int_{s}^{t} \left(\|u^{N}(r)\|_{L^\infty}^2 + 1\right)  \left\| A^\alpha S(r-s) A^{-\alpha}\psi \right\| dr \\
		&\leq C \int_{s}^{t} \left(\|u^{N}(r)\|_{L^\infty}^2 + 1\right)  (r-s)^{-\alpha} e^{-\frac{\lambda_1}{2}(r-s)} \|A^{-\alpha}\psi\| dr \\
		&\leq C \left( \int_{s}^{t} e^{-\frac{\lambda_1}{2}(r-s)} \left(\|u^{N}(r)\|_{L^\infty}^2 + 1\right)^2 dr \right)^{\frac{1}{2}}
		\left( \int_{s}^{t} e^{-\frac{\lambda_1}{2}(r-s)} (r-s)^{-2\alpha} dr \right)^{\frac{1}{2}} \|A^{-\alpha}\psi\| \\
		&\leq C \sqrt{\Gamma(1-2\alpha)} (\tfrac{2}{\lambda_{1}})^{\frac{1}{2}-\alpha} \left( \int_{s}^{t} e^{-\frac{\lambda_1}{2}(r-s)} \left(\|u^{N}(r)\|_{L^\infty}^4 + 1\right) dr \right)^{\frac{1}{2}} \|A^{-\alpha}\psi\|, 
	\end{align*}
	which, together with \eqref{Upsilon} and \eqref{semigroup1}, leads to 
	\begin{align*}
		&\|\mathcal{D}_s^{\psi}u^N(t)\| 
		\leq \|\Upsilon_s^N(t,\psi)\| + \|A^{\alpha}S(t-s)A^{-\alpha}\psi\| \\
		&\leq C \left( \int_{s}^{t} e^{-\frac{\lambda_1}{2}(r-s)} \left(\|u^{N}(r)\|_{L^\infty}^4+1\right) dr \right)^{\frac{1}{2}} \|A^{-\alpha}\psi\| + C e^{-\frac{\lambda_1}{2}(t-s)} (t-s)^{-\alpha} \|A^{-\alpha}\psi\|. 
	\end{align*}
	Finally, by taking $\psi = Q^{\frac{1}{2}} \phi_k$ and $\alpha = \frac{1 - \beta}{2} \in \left[ 0, \frac{1}{2} \right)$, and in view of \eqref{regularity uN}, we obtain
	\begin{align*}
		&\mathbb{E}\left[ \| \mathcal{D}_{s} u^N(t) \|_{\mathcal{L}_2^0}^2 \right] = \sum_{k=1}^{\infty} \mathbb{E}\left[ \| \mathcal{D}_{s}^{Q^{\frac{1}{2}} \phi_k} u^N(t) \|^2 \right] \\
		&\leq C \sum_{k=1}^{\infty} \int_{s}^{t} e^{-\frac{\lambda_1}{2} (r - s)} \left( \| u^N(r) \|_{L^4(\Omega; L^\infty)}^4 + 1 \right) \, dr \, \| A^{\frac{\beta - 1}{2}} Q^{\frac{1}{2}} \phi_k \|^2 \\
		& \quad + C \sum_{k=1}^{\infty} e^{-\lambda_1 (t - s)} (t - s)^{\beta - 1} \| A^{\frac{\beta - 1}{2}} Q^{\frac{1}{2}} \phi_k \|^2 \\
		&\leq C \frac{2}{\lambda_1} \left( 1 - e^{-\frac{\lambda_1}{2} (t - s)} \right) \| A^{\frac{\beta - 1}{2}} Q^{\frac{1}{2}} \|_{\mathcal{L}_2}^2 \left( \sup_{t \geq 0} \| u^N(t) \|_{L^4(\Omega; L^\infty)}^4 + 1 \right) + C (t - s)^{\beta - 1} \| A^{\frac{\beta - 1}{2}} Q^{\frac{1}{2}} \|_{\mathcal{L}_2}^2 \\
		& \leq C \left( (t - s)^{\beta - 1} + 1 \right) \| A^{\frac{\beta - 1}{2}} Q^{\frac{1}{2}} \|_{\mathcal{L}_2}^2,
	\end{align*}
	where the constant $C$ is independent of $t$ and $s$.
	
\end{proof}

Next, the following lemma gives the estimate of the negative norm of $F'$, which is used to obtain the weak approximation error.

\begin{Lemma}
	For the nonlinear term $F$ satisfying Assumption \ref{Asp:f}, and for $\theta \in (0, 1]$, $\eta \geq 1$, $\forall \zeta \in \dot{H}^{\theta} \cap L^\infty(\mathcal{I})$, $\psi \in L^\infty(\mathcal{I})$, it holds
	\begin{equation} \label{Est:negative norm}
		\| F'(\zeta) \psi \|_{\dot{H}^{-\eta}} \leq C \left( \max \left\{ \| \zeta \|_{L^\infty}, \| \zeta \|_{\dot{H}^{\theta}} \right\}^2 + 1 \right) \| \psi \|_{\dot{H}^{-\theta}}.
	\end{equation}
\end{Lemma}

\begin{proof}
	Since the estimate in \eqref{Est:negative norm} for $\theta \in (0, 1)$ has been proven in \cite[Lemma 3.4]{Cai2021}, we focus here on the case where $\theta = 1$. It is well known that the norms in $\dot{H}^1$ and $H_0^1(\mathcal{I})$ are equivalent, meaning that
	\begin{align*}
		&\| F'(\zeta) \psi \|_{\dot{H}^1}
		\leq C \left\| \left( -3 a_3 \zeta^2(x) + 2 a_2 \zeta(x) + a_1 \right) \psi(x) \right\|_{H_0^1} \\
		&\leq C \left\| \left( -6 a_3 \zeta(x) \zeta'(x) + 2 a_2 \zeta'(x) \right) \psi(x) \right\| + C \left\| \left( -3 a_3 \zeta^2(x) + 2 a_2 \zeta(x) + a_1 \right) \psi'(x) \right\| \\
		&\leq C \| \zeta(x) \zeta'(x) \psi(x) \| + C \| \zeta'(x) \psi(x) \| + C \| \zeta^2(x) \psi'(x) \| + C \| \zeta(x) \psi'(x) \| + C \| \psi'(x) \| \\
		&\leq C \| \zeta \|_{L^\infty} \| \zeta \|_{\dot{H}^1} \| \psi \|_{L^\infty} + C \| \zeta \|_{\dot{H}^1} \| \psi \|_{L^\infty} + C \| \zeta \|_{L^\infty}^2 \| \psi \|_{\dot{H}^1} 
		+ C \| \zeta \|_{L^\infty} \| \psi \|_{\dot{H}^1} + C \| \psi \|_{\dot{H}^1} \\
		&\leq C \left( \max \left\{ \| \zeta \|_{L^\infty}^2, \| \zeta \|_{\dot{H}^1}^2 \right\} + 1 \right) \left( \| \psi \|_{\dot{H}^1} + \| \psi \|_{L^\infty} \right).
	\end{align*}
	Accordingly, for $\eta \geq 1$, $\zeta \in \dot{H}^1 \cap L^\infty(\mathcal{I})$ and $\psi \in L^\infty(\mathcal{I})$, we have
	\begin{align*}
		&\| F'(\zeta) \psi \|_{\dot{H}^{-\eta}} 
		= \sup_{\varphi \in H} \frac{\big| \langle A^{-\frac{\eta}{2}} F'(\zeta) \psi, \varphi \rangle \big|}{\|\varphi\|}
		= \sup_{\varphi \in H} \frac{\big| \langle A^{-\frac{1}{2}} \psi, A^{\frac{1}{2}} F'(\zeta) A^{-\frac{\eta}{2}} \varphi \rangle \big|}{\|\varphi\|} \\
		&\leq C \sup_{\varphi \in H} \frac{\|\psi\|_{\dot{H}^{-1}}}{\|\varphi\|} \left( \max \left\{ \| \zeta \|_{L^\infty}^2, \| \zeta \|_{\dot{H}^1}^2 \right\} + 1 \right) 
		\left( \| A^{-\frac{\eta}{2}} \varphi \|_{\dot{H}^1} + \| A^{-\frac{\eta}{2}} \varphi \|_{L^\infty} \right) \\
		&\leq C \left( \max \left\{ \| \zeta \|_{L^\infty}^2, \| \zeta \|_{\dot{H}^1}^2 \right\} + 1 \right) \|\psi\|_{\dot{H}^{-1}},
	\end{align*}
	which completes the proof.
\end{proof}

\begin{theorem}\label{Th:semi-discrete order}
	Under Assumptions \ref{Asp:Initial value}--\ref{Asp:dissipative}, \ref{Asp:additional asp}, for any test functions $\Phi\in C_b^2(H;\mathbb{R})$, $\beta\in(0,1]$ and for any $T>0$, there exists a constant $C>0$ independent of $N$ and $T$, such that 
	\begin{equation*}
		\left| \mathbb{E}[\Phi(u(T))] - \mathbb{E}[\Phi(u^N(T))] \right|
		\leq C (1+T^{-\beta}) \lambda_{N+1}^{-\beta}. 
	\end{equation*}
\end{theorem}
\begin{proof}
	By denoting $\bar{u}(t) := u(t) - \mathcal{O}_t$ and $\bar{u}^N(t) := u^N(t) - \mathcal{O}^N_t$, we have
	\begin{align*}
		&\quad \mathbb{E}[\Phi(u(T))] - \mathbb{E}[\Phi(u^N(T))] \\
		&= \left( \mathbb{E}[\Phi(\bar{u}(T) + \mathcal{O}_T)] - \mathbb{E}[\Phi(\bar{u}^N(T) + \mathcal{O}_T)] \right) 
		+ \left( \mathbb{E}[\Phi(\bar{u}^N(T) + \mathcal{O}_T)] - \mathbb{E}[\Phi(\bar{u}^N(T) + \mathcal{O}^N_T)] \right) \\
		&=:  I_1 + I_2.
	\end{align*}
	
	We first estimate $I_2$. By the second-order Taylor expansion at $u^N(T)$, the boundedness of $\Phi''$, and \eqref{projection}, we have
	\begin{align*}
		I_2 &= \mathbb{E} \big[ \Phi'(u^N(T)) (\mathcal{O}_T - \mathcal{O}^N_T) 
		+ \int_{0}^{1} \Phi''( u^N(T) + \lambda(\mathcal{O}_T - \mathcal{O}^N_T) ) (\mathcal{O}_T - \mathcal{O}^N_T, \mathcal{O}_T - \mathcal{O}^N_T) (1 - \lambda) d\lambda \big] \\
		&\leq \left| \mathbb{E} \left[ \Phi'(u^N(T)) (\mathcal{O}_T - \mathcal{O}^N_T) \right] \right|  + C \mathbb{E} \left[ \|\mathcal{O}_T - \mathcal{O}^N_T\|^2 \right] \\
		&\leq \left| \mathbb{E} \left[ \Phi'(u^N(T)) (I - P_N) \mathcal{O}_T \right] \right| + C \lambda_{N+1}^{-\beta} \|\mathcal{O}_T\|_{L^2(\Omega; \dot{H}^{\beta})}^2.
	\end{align*}
	Through the Malliavin integration by parts \eqref{Malliavin integral by parts}, the chain rule \eqref{chain rule}, and \eqref{Malliavin of uN}, we know that
	\begin{align*}
		& \left| \mathbb{E} \left[ \Phi'(u^N(T)) (I - P_N) \mathcal{O}_T \right] \right| \\
		= \ & \left| \mathbb{E} \left[ \Phi'(u^N(T)) \int_{0}^{T} S(T - s)(I - P_N) \, dW(s) \right] \right| \\
		= \ & \left| \mathbb{E} \left[ \int_{0}^{T} \left\langle (I - P_N) S(T - s), \mathcal{D}_s \Phi'(u^N(T)) \right\rangle_{\mathcal{L}_2^0} \, ds \right] \right| \\
		= \ & \left| \mathbb{E} \left[ \int_{0}^{T} \left\langle (I - P_N) S(T - s), \Phi''(u^N(T)) \mathcal{D}_s u^N(T) \right\rangle_{\mathcal{L}_2^0} \, ds \right] \right| \\
		\leq \ & \left| \mathbb{E} \left[ \int_{0}^{T} \left\langle (I - P_N) S(T - s), \Phi''(u^N(T)) S_N(T - s) \right\rangle_{\mathcal{L}_2^0} \, ds \right] \right| \\
		&+ \left| \mathbb{E} \left[ \int_{0}^{T} \left\langle (I - P_N) S(T - s), \Phi''(u^N(T)) \int_{s}^{T} S(T - r) P_N F'(u^N(r)) \mathcal{D}_s u^N(r) \, dr \right\rangle_{\mathcal{L}_2^0} \, ds \right] \right| \\
		=: \ & J_1 + J_2.
	\end{align*}
	For $J_1$, since $A$ commutes with $Q$, by using the Cauchy--Schwarz inequality, and in view of the boundedness of $\Phi''$, we get 
	\begin{align*}
		J_1 
		&= \left| \mathbb{E} \left[ \sum_{k=1}^{\infty} \int_{0}^{T} \left\langle (I-P_N)S(T-s) Q^{\frac{1}{2}}\phi_k , \Phi''(u^N(T)) S(T-s) P_N Q^{\frac{1}{2}}\phi_k \right\rangle ds \right] \right| \\
		&= \left| \mathbb{E} \left[ \sum_{k=N+1}^{\infty} \int_{0}^{T} e^{-2\lambda_{k}(T-s)} \left\langle  Q^{\frac{1}{2}}\phi_k , \Phi''(u^N(T)) Q^{\frac{1}{2}}\phi_k \right\rangle ds \right] \right| \\
		&= \left| \mathbb{E} \left[ \sum_{k=N+1}^{\infty} \frac{1}{2\lambda_{k}} (1-e^{-2\lambda_{k}T}) \left\langle Q^{\frac{1}{2}}\phi_k , \Phi''(u^N(T)) Q^{\frac{1}{2}}\phi_k \right\rangle \right] \right| \\
		&\leq C \mathbb{E} \left[ \sum_{k=N+1}^{\infty} \lambda_{k}^{-1} \|Q^{\frac{1}{2}}\phi_k\| \|\Phi''(u^N(T))Q^{\frac{1}{2}}\phi_k\| \right] \\
		&
		\leq C \lambda_{N+1}^{-\beta} \|A^{\frac{\beta-1}{2}}Q^{\frac{1}{2}}\|_{\mathcal{L}_2(H)}^2. 
	\end{align*}
	For $J_2$, through \eqref{semigroup1}, \eqref{Est:F2}, H\"older's inequality, \eqref{regularity uN} and Lemma \ref{Le:Malliavin}, we have 
	\begin{align*}
		J_2 &\leq \int_{0}^{T} \|(I-P_N)S(T-s)\|_{\mathcal{L}_2^0} 
		\left( \int_{s}^{T} \mathbb{E} \left[ \left\| S(T-r) P_N F'(u^{N}(r)) \mathcal{D}_s u^N(r) \right\|_{\mathcal{L}_2^0} \right] dr  \right) ds \\
		&\leq C \int_{0}^{T} \left( \sum_{k=N+1}^{\infty} \lambda_{k}^{-2\beta} (T-s)^{-(\beta+1)} e^{-\lambda_{1}(T-s)} \|A^{\frac{\beta-1}{2}}Q^{\frac{1}{2}}\phi_k\|^2 \right)^{\frac{1}{2}} \\
		&\qquad \quad \left( \int_{s}^{T} \mathbb{E} \left[ \left(\sum_{k=1}^{\infty} \left\{ C_F\left(1+\|u^{N}(r)\|_{L^\infty}^2\right) \|\mathcal{D}_s^{Q^{\frac{1}{2}}\phi_k} u^N(r)\| \right\}^2 \right)^{\frac{1}{2}} \right] dr  \right) ds \\
		&\leq C \lambda_{N+1}^{-\beta} \|A^{\frac{\beta-1}{2}}Q^{\frac{1}{2}}\|_{\mathcal{L}_2(H)} \int_{0}^{T} (T-s)^{-\frac{\beta+1}{2}} e^{-\frac{\lambda_{1}}{2}(T-s)} 
		\left( \int_{s}^{T} \left( \mathbb{E}\big[\left\|\mathcal{D}_s u^N(r)\right\|_{\mathcal{L}_2^0}^2 \big]  \right)^{\frac{1}{2}}  dr  \right) ds \\
		&\leq C \lambda_{N+1}^{-\beta} \int_{0}^{T} \left(1+(T-s)^{\frac{1-\beta}{2}}\right) e^{-\frac{\lambda_{1}}{2}(T-s)} ds 
		\leq C \lambda_{N+1}^{-\beta}. 
	\end{align*}
	So we obtain $|I_2| \leq C \lambda_{N+1}^{-\beta}$.

	We then estimate $I_1$. By the first-order Taylor expansion at $\xi$ (where $\xi$ is between $\bar{u}(T)$ and $\bar{u}^N(T)$) and the boundedness of $\Phi'$, we obtain 
	\begin{align*}
		|I_1| &\leq \left| \mathbb{E} \left[ \Phi'(\xi)(\bar{u}(T)-\bar{u}^N(T)) \right] \right| 
		\leq C \mathbb{E}[\|\bar{u}(T)-\bar{u}^N(T)\|] 
		\leq C \mathbb{E}[\|e_1(T)\|] 
		+ C \mathbb{E}[\|e_2(T)\|], 
	\end{align*}
	where $e_1(t):=\bar{u}(t)-P_N\bar{u}(t)$ and $e_2(t):=\bar{u}^N(t)-P_N \bar{u}(t)$. 
	For $e_1(T)$, from \eqref{projection} we know that 
	\begin{align*}
		\mathbb{E}[\|e_1(T)\|] 
		= \mathbb{E}[\|(I-P_N)\bar{u}(T)\|]
		\leq \lambda_{N+1}^{-\beta} \mathbb{E}[\|\bar{u}(T)\|_{\dot{H}^{2\beta}}]
		\leq \lambda_{N+1}^{-\beta}  \|\bar{u}(T)\|_{L^2(\Omega;\dot{H}^{2\beta})}. 
	\end{align*}
	By using \eqref{semigroup1} with $\nu=\beta$, \eqref{semigroup_sharp} with $\rho=1$, \eqref{Holder u}, we have 
	\begin{align*}
		&\|\bar{u}(T)\|_{L^2(\Omega;\dot{H}^{2\beta})}\\
		\leq \ & \|A^{\beta}S(T)u_0\|_{L^2(\Omega;H)} + \left\| \int_{0}^{T} A^{\beta}S(T-s) F(u(T)) ds \right\|_{L^2(\Omega;H)} \\
		& + \int_{0}^{T} \left\| A^{\beta}S(T-s) (F(u(T))-F(u(s))) \right\|_{L^2(\Omega;H)} ds\\
		\leq \ & C T^{-\beta} \|u_0\|_{L^2(\Omega;H)} 
		+ C \|F(u(T))\|_{L^2(\Omega;H)} 
		+ C \int_{0}^{T} (T-s)^{-\frac{\beta}{2}} e^{-\frac{\lambda_{1}}{2}(T-s)}  ds \\
		\leq \ & C T^{-\beta} \|u_0\|_{L^2(\Omega;H)} 
		+ C \Big( 1+\sup\limits_{t\geq0}\|u(t)\|_{L^6(\Omega; L^6)}^3 \Big)
		+ C \Gamma(1-\tfrac{\beta}{2}) (\tfrac{2}{\lambda_{1}})^{1-\frac{\beta}{2}}, 
	\end{align*}
	which implies that 
	\begin{align*}
		\mathbb{E}[\|e_1(T)\|] 
		\leq C (1+T^{-\beta}) \lambda_{N+1}^{-\beta}. 
	\end{align*}
	For $e_2(t)=\bar{u}^N(t)-P_N \bar{u}(t)$, which solves 
	\begin{align*}
		\frac{d}{dt} e_2(t) = -A_N e_2(t) + P_N \left( F(\bar{u}^N(t)+\mathcal{O}^N_t) - F(\bar{u}(t)+\mathcal{O}_t) \right), \quad e_2(0)=0, 
	\end{align*}
	through the integration by parts, \eqref{Est:F1}, Young's inequality and Poincar\'e's inequality, we get
	\begin{align*}
		&\frac{d}{dt} \|e_2(t)\|^2 \\
		= & -2\|e_2(t)\|_{\dot{H}^1}^2 + 2 \left\langle e_2(t) , F(\bar{u}^N(t)+\mathcal{O}^N_t) - F(P_N\bar{u}(t)+\mathcal{O}^N_t) \right\rangle \\
		& + 2 \left\langle e_2(t) ,  F(P_N\bar{u}(t)+\mathcal{O}^N_t) - F(\bar{u}(t)+\mathcal{O}_t) \right\rangle \\
		\leq & -2\|e_2(t)\|_{\dot{H}^1}^2 
		+ 2 L_F \|e_2(t)\|^2 
		+ \tfrac{\lambda_{1}-L_F}{\lambda_{1}} \|e_2(t)\|_{\dot{H}^1}^2 + \tfrac{\lambda_{1}}{\lambda_{1}-L_F} \|A^{-\frac{1}{2}} (F(P_Nu(t))-F(u(t))) \|^2 \\
		\leq & -(\lambda_{1}-L_F) \|e_2(t)\|^2 
		+ \tfrac{\lambda_{1}}{\lambda_{1}-L_F} \|F(P_Nu(t))-F(u(t))\|_{\dot{H}^{-1}}^2. 
	\end{align*}
	Note that $\lambda_{1}-L_F>0$ ensured by Assumption \ref{Asp:dissipative}, then by Gronwall's inequality, we obtain 
	\begin{align*}
		\|e_2(t)\|^2 
		\leq \tfrac{\lambda_{1}}{\lambda_{1}-L_F} \int_{0}^{t} e^{-(\lambda_{1}-L_F)(t-s)} 
		\|F(P_Nu(s))-F(u(s))\|_{\dot{H}^{-1}}^2 ds. 
	\end{align*}
	It follows from Taylor's expansion and \eqref{Est:negative norm} that 
	\begin{align*}
		&\|F(P_Nu(s))-F(u(s))\|_{\dot{H}^{-1}}^2 
		\leq \int_{0}^{1} \left\| F'\left(u(s)+\lambda(P_Nu(s)-u(s))\right) (P_Nu(s)-u(s)) \right\|_{\dot{H}^{-1}}^2 d\lambda \\
		&\leq C \int_{0}^{1} \left(\max\left\{ \|u(s)\|_{L^\infty}^4, \|P_Nu(s)\|_{L^\infty}^4, \|u(s)\|_{\dot{H}^{\theta}}^4 \right\} + 1 \right) \|P_Nu(s)-u(s)\|_{\dot{H}^{-\theta}}^2 d\lambda. 
	\end{align*}
	By taking \(\theta = \beta\) in the above estimate, and then applying H\"older's inequality, \eqref{projection} with \(s = \beta\), \eqref{regularity u}, and \eqref{regularity uN}, we obtain
	\begin{align*}
		&\mathbb{E} \left[\|F(P_Nu(s)) - F(u(s))\|_{\dot{H}^{-1}}^2\right] \\
		\leq \ & C \left( \mathbb{E} \left[ \left( \max \left\{ \|u(s)\|_{L^\infty}^8, \|P_Nu(s)\|_{L^\infty}^8, \|u(s)\|_{\dot{H}^{\beta}}^8 \right\} + 1 \right) \right] \right)^{\frac{1}{2}}
		\left( \mathbb{E} \left[ \|P_Nu(s) - u(s)\|_{\dot{H}^{-\beta}}^4 \right] \right)^{\frac{1}{2}} \\
		\leq \ & C \sup_{t \geq 0} \left( \|u(t)\|_{L^8(\Omega; L^\infty)}^4 + \|P_Nu(t)\|_{L^8(\Omega; L^\infty)}^4 + \|u(t)\|_{L^8(\Omega; \dot{H}^{\beta})}^4 + 1 \right) 
		\lambda_{N+1}^{-2\beta} \sup_{t \geq 0} \|u(t)\|_{L^4(\Omega; \dot{H}^{\beta})}^2 \\
		\leq \ & C \lambda_{N+1}^{-2\beta},
	\end{align*}
	which implies that
	\begin{align*}
		\mathbb{E}[\|e_2(T)\|] &\leq \left( \mathbb{E}[\|e_2(T)\|^2] \right)^{\frac{1}{2}} 
		\leq \Big( \tfrac{\lambda_1}{\lambda_1 - L_F} \int_0^t e^{-(\lambda_1 - L_F)(t-s)} \mathbb{E} [\|F(P_Nu(s)) - F(u(s))\|_{\dot{H}^{-1}}^2 ] ds \Big)^{\frac{1}{2}} \\
		&\leq C \lambda_{N+1}^{-\beta} \Big( \tfrac{\lambda_1}{\lambda_1 - L_F} \int_0^t e^{-(\lambda_1 - L_F)(t-s)} ds \Big)^{\frac{1}{2}} \leq C \lambda_{N+1}^{-\beta}.
	\end{align*}
	Finally, combining all the above estimates yields the desired conclusion.
\end{proof}

\section{Error estimate for the full discretization}\label{Sec:Full discrete}
By $\tau$ we denote the temporal step size. In the following we suppose that $\tau\in(0,\tau_0]$, where $\tau_0$ is an arbitrary parameter. 
Once the value of $\tau$ is fixed, by $t_k = k \tau$, $k \in \mathbb{N}\cup\{0\}$, we denote the time grid-points. 
In order to discretize \eqref{mild uN} in the
time direction, we use tamed accelerated exponential Euler method to construct the fully discrete scheme: $u^{N}_0 = P_N u_0$ and for $k \in \mathbb{N}\cup\{0\}$, 
\begin{equation}\label{scheme1}
	u^{N}_{k+1} = S_N(\tau)u^{N}_{k}
	+ \frac{A_N^{-1}(I-S_N(\tau))F_N(u^{N}_{k})}{ \ 1 + \tau^{\beta}\|u^{N}_{k}\|_{L^\infty}^6 + \tau^{\beta}\|u^{N}_{k}\|_{\dot{H}^{\beta}}^6 \ }
	+ \int_{t_k}^{t_{k+1}} S_N(t_{k+1}-s) P_N dW(s). 
\end{equation}

\subsection{Boundedness of moments of numerical solutions}
By introducing $\kappa(t) := \tau \lfloor t/\tau \rfloor$ for $t\geq0$, where $\lfloor \cdot \rfloor$ is the floor function, namely, 
\begin{align*}
	\kappa(t)=t_k=k\tau, \quad \text{for } t \in [t_k,t_{k+1}), \  k \in \mathbb{N}\cup\{0\}, 
\end{align*}
the fully discrete scheme \eqref{scheme1} can be reformulated as 
\begin{equation}\label{scheme2}
	u^{N,\tau}_{t} = S_N(t) u^{N}_{0} 
	+ \int_{0}^{t} G(s) S_N(t-s)F_N(u^{N,\tau}_{\kappa(s)}) ds 
	+ \mathcal{O}_t^N, 
\end{equation}
where $u_0^N=P_Nu_0$, and $G(t)$ is a step function given by
\begin{equation}\label{G(t)}
	G(t):=\frac{1}{1+\tau^{\beta}\|u^{N,\tau}_{\kappa(t)}\|_{L^\infty}^6+\tau^{\beta}\|u^{N,\tau}_{\kappa(t)}\|_{\dot{H}^{\beta}}^6}.  
\end{equation}
Obviously, $u^{N,\tau}_{t_k}=u^{N}_k$ for any $k\in\mathbb{N}\cup\{0\}$, where $u^{N,\tau}_{t_k}$ and $u^{N}_k$ are given by \eqref{scheme2} and \eqref{scheme1}, respectively.

Next we give the boundedness of moments of fully discrete solutions. 
\begin{theorem}\label{Th:moment bound v}
	Under Assumptions \ref{Asp:Initial value}--\ref{Asp:f}, \ref{Asp:additional asp}, for $p \geq 2$ and $\beta \in (0,1]$, there exists a constant $C(u_0,p,Q,\tau_0)>0$, such that 
	\begin{equation}\label{regularity v}
		\sup\limits_{N\in\mathbb{N}, \tau\in(0,\tau_0], t\geq0} \|u^{N,\tau}_{t}\|_{L^p(\Omega;L^\infty)} \leq C(u_0,p,Q,\tau_0)
		< \infty. 
	\end{equation}
\end{theorem}
\begin{proof}
	Denote $\bar{u}^{N,\tau}_{t} := u^{N,\tau}_{t}-\mathcal{O}_t^N$, then $\bar{u}^{N,\tau}_{t}$ solves 
	\begin{equation*}
		\frac{d}{dt} \bar{u}^{N,\tau}_{t} = -A_N\bar{u}^{N,\tau}_{t} + G(t)  F_N(u^{N,\tau}_{\kappa(t)}), \quad \bar{u}^{N,\tau}_0=u^{N}_0. 
	\end{equation*}
	Through the integration by parts, Young's inequality and properties of the nonlinear term $F$, we have 
	\begin{equation}\label{Est:vbar L2}
		\begin{aligned}
			\frac{1}{2} \frac{d}{dt} \|\bar{u}^{N,\tau}_{t}\|^2 
			&= -\|\bar{u}^{N,\tau}_{t}\|_{\dot{H}^1}^2 + G(t) \langle \bar{u}^{N,\tau}_{t} , F(\bar{u}^{N,\tau}_{t}+\mathcal{O}_t^N) \rangle 
			- G(t) \langle \bar{u}^{N,\tau}_{t} ,  F(u^{N,\tau}_{t})-F(u^{N,\tau}_{\kappa(t)}) \rangle \\
			&\leq -\|\bar{u}^{N,\tau}_{t}\|_{\dot{H}^1}^2 + G(t) \left( -(a_3-\epsilon)\|\bar{u}^{N,\tau}_{t}\|_{L^4}^4 + C(\epsilon) (\|\mathcal{O}_t^N\|_{L^4}^4+1) \right) \\
			&\quad + \frac{1}{2}\|\bar{u}^{N,\tau}_{t}\|^2 + \frac{1}{2}\left(G(t) \|F(u^{N,\tau}_{t})-F(u^{N,\tau}_{\kappa(t)})\|\right)^2. 
		\end{aligned}
	\end{equation}
	From 
	\begin{equation*}
		u^{N,\tau}_{t} = S_N(t-\kappa(t)) u^{N,\tau}_{\kappa(t)} 
		+ \int_{\kappa(t)}^{t} G(s)S_N(t-s)F_N(u^{N,\tau}_{\kappa(s)}) ds + \mathcal{O}_t^N - S_N(t-\kappa(t)) \mathcal{O}_{\kappa(t)}^N, 
	\end{equation*}
	and $\|S(t)u\|_{L^\infty}\leq C\|u\|_{L^\infty}$, \eqref{Est:PN S 1} with $\nu=0$ and $\tau^\beta\|u^{N,\tau}_{\kappa(t)}\|_{L^\infty}^6 \geq \tau^{\frac{\beta}{2}}\|u^{N,\tau}_{\kappa(t)}\|_{L^\infty}^3-\frac{1}{4}$, we obtain  
	\begin{equation}\label{Est:vbar L8}
		\begin{aligned}
			\|u^{N,\tau}_{t}\|_{L^{\infty}} 
			&\leq C\|u^{N,\tau}_{\kappa(t)}\|_{L^{\infty}} 
			+ C \int_{\kappa(t)}^{t} G(\kappa(t)) (t-s)^{-\frac{1}{4}} \|F(u^{N,\tau}_{\kappa(t)})\| ds + C\|\mathcal{O}_t^N\|_{L^{\infty}} + C\|\mathcal{O}_{\kappa(t)}^N\|_{L^{\infty}} \\  
			&\leq C \|u^{N,\tau}_{\kappa(t)}\|_{L^{\infty}} 
			+ \frac{C\tau^{\frac{3}{4}} (\|u^{N,\tau}_{\kappa(t)}\|_{L^{6}}^3 + 1)}{1+\tau^{\beta}\|u^{N,\tau}_{\kappa(t)}\|_{L^\infty}^6+\tau^{\beta}\|u^{N,\tau}_{\kappa(t)}\|_{\dot{H}^{\beta}}^6} + C\|\mathcal{O}_t^N\|_{L^{\infty}} + C\|\mathcal{O}_{\kappa(t)}^N\|_{L^{\infty}} \\
			&\leq C \|u^{N,\tau}_{\kappa(t)}\|_{L^{\infty}} 
			+ \frac{C\tau^{\frac{3}{4}} (\|u^{N,\tau}_{\kappa(t)}\|_{L^{6}}^3 + 1)}{1+\tau^{\frac{\beta}{2}}\|u^{N,\tau}_{\kappa(t)}\|_{L^\infty}^3-\frac{1}{4}} + C\|\mathcal{O}_t^N\|_{L^{\infty}} + C\|\mathcal{O}_{\kappa(t)}^N\|_{L^{\infty}} \\
			&\leq C \|u^{N,\tau}_{\kappa(t)}\|_{L^{\infty}} + C \tau^{\frac{3}{4}-\frac{\beta}{2}} + C\|\mathcal{O}_t^N\|_{L^{\infty}} + C\|\mathcal{O}_{\kappa(t)}^N\|_{L^{\infty}}.\end{aligned}
	\end{equation}
	Furthermore, by \eqref{semigroup2} with $\rho=\frac{\beta}{2}$ we have 
	\begin{equation}\label{Est:vbar Holder}
		\begin{aligned}
			\|u^{N,\tau}_{t}-u^{N,\tau}_{\kappa(t)}\| 
			&\leq \left\| A^{-\frac{\beta}{2}} (I-S(t-\kappa(t))) A^{\frac{\beta}{2}}u^{N,\tau}_{\kappa(t)} \right\|  
			+ C \frac{\tau (\|u^{N,\tau}_{\kappa(t)}\|_{L^{6}}^3 + 1)}{1+\tau^{\beta}\|u^{N,\tau}_{\kappa(t)}\|_{L^\infty}^6+\tau^{\beta}\|u^{N,\tau}_{\kappa(t)}\|_{\dot{H}^{\beta}}^6} \\ 
			&\quad + \Big\| \int_{\kappa(t)}^{t} S_N(t-s) dW(s) \Big\| \\
			&\leq C \tau^{\frac{\beta}{2}}  \|u^{N,\tau}_{\kappa(t)}\|_{\dot{H}^{\beta}}
			+ C \tau^{1-\frac{\beta}{2}}
			+ \Big\| \int_{\kappa(t)}^{t} S_N(t-s) dW(s) \Big\|. 
		\end{aligned}
	\end{equation} 
	For convenience, we introduce the following notations 
	\begin{equation}\label{notation}
		\mathbb{O}_t := \|\mathcal{O}^N_{\kappa(t)}\|_{L^{\infty}} + \|\mathcal{O}^N_{t}\|_{L^{\infty}}  
		\quad  \text{and} \quad  \mathbb{W}_t := \Big\| \int_{\kappa(t)}^{t} S_N(t-s) dW(s) \Big\|.
	\end{equation}
	Combining with \eqref{Est:vbar L8}, \eqref{Est:vbar Holder}, the definition of $G(t)$, and the fact $\beta\in(0,1]$, we obtain 
	\begin{equation}\label{Est:G F-F}
		\begin{aligned}
			&G(t) \|F(u^{N,\tau}_{t})-F(u^{N,\tau}_{\kappa(t)})\| \\
			\leq \ & C_F G(t) \left( 1+\|u^{N,\tau}_t\|_{L^\infty}^2+\|u^{N,\tau}_{\kappa(t)}\|_{L^\infty}^2 \right) \|u^{N,\tau}_{t}-u^{N,\tau}_{\kappa(t)}\| \\
			\leq \ & C G(t) \left( 1+ \|u^{N,\tau}_{\kappa(t)}\|_{L^{\infty}}^2 + \mathbb{O}_t^2
			+ \tau^{\frac{3}{2}-\beta} \right) 
			\left( \tau^{\frac{\beta}{2}} \|u^{N,\tau}_{\kappa(t)}\|_{\dot{H}^{\beta}}
			+ \tau^{1-\frac{\beta}{2}}  
			+ \mathbb{W}_t \right)	\\
			\leq \ & C G(t) \tau^{\frac{\beta}{2}} \left( 1 + C(\tau_0) + \|u^{N,\tau}_{\kappa(t)}\|_{L^{\infty}}^3 + \|u^{N,\tau}_{\kappa(t)}\|_{\dot{H}^{\beta}}^3 + \mathbb{O}_t^3 \right) \\
			&+ C G(t) \tau^{1-\frac{\beta}{2}} \left( 1 + C(\tau_0) + \|u^{N,\tau}_{\kappa(t)}\|_{L^{\infty}}^2 + \mathbb{O}_t^2 \right)  
			+ C G(t) \left( 1 + C(\tau_0) + \|u^{N,\tau}_{\kappa(t)}\|_{L^{\infty}}^2 + \mathbb{O}_t^2 \right) \mathbb{W}_t \\
			\leq \ & C(\tau_0) ( 1 + \mathbb{O}_t^3 )
			+ C(\tau_0) ( \tau^{-\frac{\beta}{3}} + \mathbb{O}_t^2 ) \mathbb{W}_t. 
		\end{aligned}
	\end{equation}
	Substituting \eqref{Est:G F-F} into \eqref{Est:vbar L2} and using Poincar\'e's inequality, yields that 
	\begin{align*}
		\frac{1}{2} \frac{d}{dt} \|\bar{u}^{N,\tau}_{t}\|^2 
		\leq -\frac{\lambda_{1}}{2}\|\bar{u}^{N,\tau}_{t}\|^2 + C (\|\mathcal{O}_t^N\|_{L^4}^4+1) + C(\tau_0) (1+\mathbb{O}_t^6)
		+ C(\tau_0) ( \tau^{-\frac{2\beta}{3}} + \mathbb{O}_t^4 ) \mathbb{W}_t^2. 
	\end{align*}
	By Gronwall's inequality, H\"older's inequality, \eqref{Holder O} with $\gamma=0$, then taking $p/2$-moment, $p\geq2$, and recalling notations $\mathbb{O}_t$ and $\mathbb{W}_t$ given by \eqref{notation}, we get 
	\begin{equation}\label{Est:vbar Lp}
		\begin{aligned}
			&\mathbb{E}[\|\bar{u}^{N,\tau}_{t}\|^p] 
			\leq Ce^{-\frac{\lambda_{1}pt}{2}} \mathbb{E}[\|u^{N}_0\|^p] 
			+ C\int_{0}^{t} e^{-\lambda_{1}(t-s)} \left(\mathbb{E}[(1+\mathbb{O}_s^6)^{\frac{p}{2}}]
			+ \mathbb{E}\big[(\tau^{-\frac{2\beta}{3}} + \mathbb{O}_s^4)^{\frac{p}{2}} \mathbb{W}_s^p\big] \right) ds  \\ 
			&\leq Ce^{-\frac{\lambda_{1}pt}{2}} \mathbb{E}[\|u^{N}_0\|^p] 
			+ C\int_{0}^{t} e^{-\lambda_{1}(t-s)} \left(1 + \tau^{-\frac{p\beta}{3}} \mathbb{E}[\mathbb{W}_s^p] + \left(\mathbb{E}[\mathbb{O}_s^{4p}]\right)^{\frac{1}{2}} \left(\mathbb{E}[\mathbb{W}_s^{2p}]\right)^{\frac{1}{2}} \right) ds  \\ 
			&\leq Ce^{-\frac{\lambda_{1}pt}{2}} \mathbb{E}[\|u^{N}_0\|^p] 
			+ C\int_{0}^{t} e^{-\lambda_{1}(t-s)} (1 + \tau^{-\frac{p\beta}{3}} \tau^{\frac{p\beta}{2}} + \tau^{\frac{p\beta}{2}} ) ds \\
			&\leq C \left(e^{-\frac{\lambda_{1}pt}{2}} \mathbb{E}[\|u^{N}_0\|^p] 
			+ 1\right). 
		\end{aligned}
	\end{equation}

	Next we estimate $\mathbb{E}[\|\bar{u}^{N,\tau}_{t}\|_{\dot{H}^1}^p]$.  
	By the identity $\frac{1}{2} \frac{d}{dt} \|\bar{u}^{N,\tau}_{t}\|_{\dot{H}^1}^2 = \langle A\bar{u}^{N,\tau}_{t} , \frac{d}{dt}\bar{u}^{N,\tau}_{t} \rangle$, we have 
	\begin{equation}\label{Est:vbar H1}
		\begin{aligned}
			\frac{1}{2} \frac{d}{dt} \|\bar{u}^{N,\tau}_{t}\|_{\dot{H}^1}^2 
			&= -\|A\bar{u}^{N,\tau}_{t}\|^2
			+ \langle A\bar{u}^{N,\tau}_{t},G(t)F(u^{N,\tau}_{t}) \rangle 
			- G(t) \langle A\bar{u}^{N,\tau}_{t} , F(u^{N,\tau}_{t})-F(u^{N,\tau}_{\kappa(t)}) \rangle \\
			&=: -\|A\bar{u}^{N,\tau}_{t}\|^2 + \mathbb{K}_1 + \mathbb{K}_2. 
		\end{aligned}
	\end{equation}
	For $\mathbb{K}_1$, by the integration by parts, Young's inequality, and the Gagliardo--Nirenberg inequality, 
	\begin{equation*}
		\|x\|_{L^p} \leq C \|Ax\|^\alpha \|x\|^{1-\alpha}, \quad \text{where } \alpha=\frac{p-2}{4p}\in(0,1], \ C>0, 
	\end{equation*}
	we know that 
	\begin{equation}\label{Est:K1}
		\begin{aligned}
			&\mathbb{K}_1 = G(t) \left\langle A\bar{u}^{N,\tau}_{t},-a_3(\bar{u}^{N,\tau}_{t}+\mathcal{O}_t^N)^3+a_2(\bar{u}^{N,\tau}_{t}+\mathcal{O}_t^N)^2+a_1(\bar{u}^{N,\tau}_{t}+\mathcal{O}_t^N)+a_0 \right\rangle \\
			&\leq -a_3G(t) \langle (\partial_{x} \bar{u}^{N,\tau}_{t})^2,3(\bar{u}^{N,\tau}_{t})^2 \rangle
			+ \epsilon G(t) \|A\bar{u}^{N,\tau}_{t}\|^2  \\
			&\quad + C(\epsilon)G(t) \left( 1 + \|\bar{u}^{N,\tau}_{t}\|_{L^4}^4\|\mathcal{O}_t^N\|_{L^\infty}^2
			+ \|\bar{u}^{N,\tau}_{t}\|^2\|\mathcal{O}_t^N\|_{L^\infty}^4 + \|\mathcal{O}_t^N\|_{L^\infty}^6 + \|\bar{u}^{N,\tau}_{t}\|_{L^4}^4 \right) \\
			&\leq \epsilon\|A\bar{u}^{N,\tau}_{t}\|^2 
			+ C(\epsilon) \left( 1 + \|A\bar{u}^{N,\tau}_{t}\|^{\frac{1}{2}}\|\bar{u}^{N,\tau}_{t}\|^{\frac{7}{2}} (\|\mathcal{O}_t^N\|_{L^\infty}^2+1)
			+ \|\bar{u}^{N,\tau}_{t}\|^2\|\mathcal{O}_t^N\|_{L^\infty}^4 + \|\mathcal{O}_t^N\|_{L^\infty}^6 \right) \\
			&\leq 2\epsilon\|A\bar{u}^{N,\tau}_{t}\|^2 
			+ C(\epsilon) \left( 1 + \|\bar{u}^{N,\tau}_{t}\|^{\frac{14}{3}} (\|\mathcal{O}_t^N\|_{L^\infty}^{\frac{8}{3}}+1)
			+ \|\bar{u}^{N,\tau}_{t}\|^2 \|\mathcal{O}_t^N\|_{L^\infty}^4 + \|\mathcal{O}_t^N\|_{L^\infty}^6 \right). 
		\end{aligned}
	\end{equation}
	For $\mathbb{K}_2$, by Young's inequality and \eqref{Est:G F-F} we have 
	\begin{equation}\label{Est:K2}
		\begin{aligned}
			&\mathbb{K}_2 = - G(t) \langle A\bar{u}^{N,\tau}_{t} , F(u^{N,\tau}_{t})-F(u^{N,\tau}_{\kappa(t)}) \rangle \\
			&\leq \epsilon \|A\bar{u}^{N,\tau}_{t}\|^2 + C(\epsilon) (G(t))^2 \|F(u^{N,\tau}_{t})-F(u^{N,\tau}_{\kappa(t)})\|^2 \\
			&\leq \epsilon \|A\bar{u}^{N,\tau}_{t}\|^2 + C(\epsilon,\tau_0) (1+\mathbb{O}_t^6)
			+ C(\epsilon,\tau_0) ( \tau^{-\frac{2\beta}{3}} + \mathbb{O}_t^4 ) \mathbb{W}_t^2. 
		\end{aligned}
	\end{equation}
	Substituting \eqref{Est:K1} and \eqref{Est:K2} into \eqref{Est:vbar H1}, then taking $\epsilon=\frac{1}{6}$ and using Poincar\'e's inequality, we get 
	\begin{align*} 
		\frac{1}{2} \frac{d}{dt} \|\bar{u}^{N,\tau}_{t}\|_{\dot{H}^1}^2 
		&\leq (3\epsilon-1) \lambda_{1} \|\bar{u}^{N,\tau}_{t}\|_{\dot{H}^1}^2 + C(\epsilon) \left( 1 + \|\bar{u}^{N,\tau}_{t}\|^{\frac{14}{3}} (\|\mathcal{O}_t^N\|_{L^\infty}^{\frac{8}{3}}+1)
		+ \|\bar{u}^{N,\tau}_{t}\|^2\|\mathcal{O}_t^N\|_{L^\infty}^4 \right) \\
		&\quad + C(\epsilon) \|\mathcal{O}_t^N\|_{L^\infty}^6
		+ C(\epsilon,\tau_0) (1+\mathbb{O}_t^6)
		+ C(\epsilon,\tau_0) ( \tau^{-\frac{2\beta}{3}} + \mathbb{O}_t^4 ) \mathbb{W}_t^2 \\
		&\leq -\frac{\lambda_{1}}{2} \|\bar{u}^{N,\tau}_{t}\|_{\dot{H}^1}^2 + C \left( 1 + \|\bar{u}^{N,\tau}_{t}\|^{\frac{14}{3}} (\|\mathcal{O}_t^N\|_{L^\infty}^{\frac{8}{3}}+1)
		+ \mathbb{O}_t^6 \right)
		+ C ( \tau^{-\frac{2\beta}{3}} + \mathbb{O}_t^4 ) \mathbb{W}_t^2. 
	\end{align*}
	By Gronwall's inequality, \eqref{Est:vbar Lp}, H\"older's inequality, then taking $p/2$-moment, we obtain 
	\begin{align*}
		&\quad \mathbb{E}[\|\bar{u}^{N,\tau}_{t}\|_{\dot{H}^1}^p] \\
		&\leq C e^{-\frac{\lambda_{1}pt}{2}} \mathbb{E}[\|u^{N}_0\|_{\dot{H}^1}^p] 
		+ C \int_{0}^{t} e^{-\frac{\lambda_{1}p(t-s)}{2}} \left( 1 + \left(\mathbb{E}[\|\bar{u}^{N,\tau}_{s}\|^{\frac{14p}{3}}]\right)^{\frac{1}{2}} \left(\left(\mathbb{E}[\|\mathcal{O}_s^N\|_{L^\infty}^{\frac{8p}{3}}]\right)^{\frac{1}{2}}+1\right) \right) ds \\
		&\quad + C\int_{0}^{t} e^{-\frac{\lambda_{1}p(t-s)}{2}} \left( \mathbb{E}[\mathbb{O}_s^{3p}] + \tau^{-\frac{p\beta}{3}} \mathbb{E}[\mathbb{W}_s^p] + \left(\mathbb{E}[\mathbb{O}_s^{4p}]\right)^{\frac{1}{2}} \left(\mathbb{E}[\mathbb{W}_s^{2p}]\right)^{\frac{1}{2}} \right) ds \\
		&\leq C e^{-\frac{\lambda_{1}pt}{2}} \mathbb{E}[\|u^{N}_0\|_{\dot{H}^1}^p] + 
		C(p,Q) \left(\left(\mathbb{E}[\|u^{N}_0\|^{\frac{14p}{3}}]\right)^{\frac{1}{2}} + \tau^{-\frac{\beta p}{3}} \tau^{\frac{\beta p}{2}} + \tau^{\frac{\beta p}{2}} + 1 \right) \\
		&\leq C(u_0,p,Q,\tau_0) < \infty. 
	\end{align*}
	Finally, through \eqref{Est:PN S 1}, the Sobolev embedding inequality $\dot{H}^1 \hookrightarrow L^\infty(\mathcal{I})$ and the above estimate, 
	\begin{equation*}
		\|u^{N,\tau}_{t}\|_{L^p(\Omega;L^\infty)} 
		\leq \|\bar{u}^{N,\tau}_{t}\|_{L^p(\Omega;L^\infty)} + \|\mathcal{O}^N_{t}\|_{L^p(\Omega;L^\infty)} \leq C(u_0,p,Q,\tau_0) 
		< \infty, 
	\end{equation*}
	which completes the proof. 
\end{proof}

By using Theorem \ref{Th:moment bound v}, it is easy to derive the regularity properties of $u^{N,\tau}_t$. 
\begin{Corollary}\label{Corollary 2}
	Let Assumptions \ref{Asp:Initial value}--\ref{Asp:f}, and \ref{Asp:additional asp} hold. Then for $p \geq 2$ and $\beta \in (0,1]$, there exists a constant $C>0$ such that 
	\begin{equation}\label{regularity v1}
		\sup\limits_{N \in \mathbb{N}, \tau\in(0,\tau_0], t\geq0} \|u^{N,\tau}_{t}\|_{L^p(\Omega; \dot{H}^\beta)}  \leq C < \infty. 
	\end{equation}
	Furthermore, there exists a constant $C$ independent of $N$ and $\tau$, such that for $\gamma \in [0, \beta]$, 
	\begin{equation}\label{Holder v}
		\|u^{N,\tau}_t-u^{N,\tau}_s\|_{L^p(\Omega;\dot{H}^{-\gamma})} 
		\leq C (t-s)^{\frac{\beta+\gamma}{2}}, \quad 0 \leq s < t.  
	\end{equation} 
\end{Corollary}
\begin{proof}
	The estimate \eqref{regularity v1} is standard, so we only focus on \eqref{Holder v}. We first note that 
	\begin{align*}
		u^{N,\tau}_{t}-u^{N,\tau}_{s} = (S_N(t-s)-I) u^{N,\tau}_{s} 
		+ \int_{s}^{t}  G(r) S_N(t-r)F_N(u^{N,\tau}_{\kappa(r)}) dr 
		+ \int_{s}^{t} S_N(t-r) dW(r). 
	\end{align*}
	Through \eqref{semigroup2} we get 
	\begin{align*}
		\|(S(t-s)-I)u^{N,\tau}_{s}\|_{L^p(\Omega;\dot{H}^{-\gamma})}
		&= \|A^{-\frac{\beta+\gamma}{2}} (S_N(t-s)-I) A^{\frac{\beta}{2}}u^{N,\tau}_{s}\|_{L^p(\Omega;H)} \\
		&\leq C (t-s)^{\frac{\beta+\gamma}{2}} \|u^{N,\tau}_s\|_{L^p(\Omega;\dot{H}^{\beta})}.
	\end{align*}
	Since $G(t) \in (0,1]$, we have 
	\begin{align*}
		&\left\| \int_{s}^{t}  G(r)S_N(t-r)F_N(u^{N,\tau}_{\kappa(r)}) dr \right\|_{L^p(\Omega;\dot{H}^{-\gamma})} \\
		\leq \ & \int_{s}^{t} \|F(u^{N,\tau}_{\kappa(r)})\|_{L^p(\Omega;H)} dr 
		\leq C (t-s) \left(1+\sup\limits_{t \geq 0}\|u^{N,\tau}_t\|_{L^{3p}(\Omega;L^6)}^3\right). 
	\end{align*}
	By the Burkholder--Davis--Gundy inequality and It\^o's isometry, it is easy to derive
	\begin{equation}\label{Holder O}
		\left\| \int_{s}^{t} S_N(t-r) dW(r) \right\|_{L^p(\Omega;\dot{H}^{-\gamma})} \leq C (t-s)^{\frac{\beta+\gamma}{2}}, \quad \forall 0 \leq s < t,  
	\end{equation}
	which completes the proof. 
\end{proof}

\subsection{Weak convergence rate of the full discretization}
Next we analyze the weak convergence of our proposed scheme in $[0,T]$ for any $T>0$. 
\begin{proof}[Proof of Theorem \ref{Th:full-discrete order}]
	Since the estimate of $\left| \mathbb{E}[\Phi(u(T))] - \mathbb{E}[\Phi(u^N(T))] \right|$ has already been done in Theorem \ref{Th:semi-discrete order}, it is sufficient to consider the weak error between $u^N(T)$ and $u^{N,\tau}_T$ for any $T>0$. By Taylor's expansion and the boundedness of $\Phi'$, we get 
	\begin{align*}
		&\left| \mathbb{E}[\Phi(u^N(T))] - \mathbb{E}[\Phi(u^{N,\tau}_T)] \right|
		= \left| \mathbb{E}[\Phi(\bar{u}^N(T)+\mathcal{O}_T^N)] - \mathbb{E}[\Phi(\bar{u}^{N,\tau}_T+\mathcal{O}_T^N)] \right| \\
		&= \left| \mathbb{E} \left[ \int_{0}^{1} \Phi'(u^{N,\tau}_T+\lambda(u^N(T)-u^{N,\tau}_T)) (\bar{u}^N(T)-\bar{u}^{N,\tau}_T) d\lambda \right] \right|
		\leq C \|\bar{u}^N(T)-\bar{u}^{N,\tau}_T\|_{L^2(\Omega;H)}. 
	\end{align*}
	We denote $\mathcal{E}_t := \bar{u}^N(t)-\bar{u}^{N,\tau}_t$, which solves 
	\begin{align*}
		\frac{d}{dt} \mathcal{E}_t = -A_N \mathcal{E}_t + F_N(\bar{u}^N(t)+\mathcal{O}^N_t) - G(t)F_N(\bar{u}^{N,\tau}_{\kappa(t)}+\mathcal{O}_{\kappa(t)}^N), \quad \mathcal{E}_0=0. 
	\end{align*} 
	Through the identity $\frac{1}{2} \frac{d}{dt} \|\mathcal{E}_t\|^2 = \langle \mathcal{E}_t,\frac{d}{dt}\mathcal{E}_t \rangle$ and the integration by parts, we obtain 
	\begin{align*}
		&\quad \frac{1}{2} \frac{d}{dt} \|\mathcal{E}_t\|^2 + \|\mathcal{E}_t\|_{\dot{H}^1}^2 \\
		&= \left\langle \mathcal{E}_t , F(\bar{u}^N(t)+\mathcal{O}^N_t) - F(\bar{u}^{N,\tau}_{t}+\mathcal{O}_{t}^N) \right\rangle 
		+ \left\langle \mathcal{E}_t , F(\bar{u}^{N,\tau}_{t}+\mathcal{O}_{t}^N) - F(\bar{u}^{N,\tau}_{\kappa(t)}+\mathcal{O}_{\kappa(t)}^N) \right\rangle \\
		&\quad + \left\langle \mathcal{E}_t , F(\bar{u}^{N,\tau}_{\kappa(t)}+\mathcal{O}_{\kappa(t)}^N) - G(t)F(\bar{u}^{N,\tau}_{\kappa(t)}+\mathcal{O}_{\kappa(t)}^N) \right\rangle \\
		&=: K_1 + K_2 + K_3. 
	\end{align*}
	For $K_1$, from \eqref{Est:F1}, we know that 
	\begin{equation}\label{proof K1}
		\begin{aligned}
			&K_1 = \left\langle \bar{u}^N(t)-\bar{u}^{N,\tau}_t , F(\bar{u}^N(t)+\mathcal{O}^N_t) - F(\bar{u}^{N,\tau}_{t}+\mathcal{O}_{t}^N) \right\rangle \\
			&\leq L_F \|\bar{u}^N(t)-\bar{u}^{N,\tau}_t\|^2 = L_F \|\mathcal{E}_t\|^2. 
		\end{aligned}
	\end{equation}
	For $K_3$, by Young's inequality and $\frac{\tau a}{1+\tau a} \leq \tau a$ for $a \geq 0$, we have 
	\begin{equation}\label{proof K3}
		\begin{aligned}
			&K_3 = \langle \mathcal{E}_t , F(\bar{u}^{N,\tau}_{\kappa(t)}+\mathcal{O}_{\kappa(t)}^N) \rangle \frac{\tau^{\beta}\|u^{N,\tau}_{\kappa(t)}\|_{L^\infty}^6+\tau^{\beta}\|u^{N,\tau}_{\kappa(t)}\|_{\dot{H}^{\beta}}^6}{1+\tau^{\beta}\|u^{N,\tau}_{\kappa(t)}\|_{L^\infty}^6+\tau^{\beta}\|u^{N,\tau}_{\kappa(t)}\|_{\dot{H}^{\beta}}^6} \\
			&\leq \epsilon \|\mathcal{E}_t\|^2 + \frac{2\tau^{2\beta}}{4\epsilon} \|F(u^{N,\tau}_{\kappa(t)})\|^2  \left( \|u^{N,\tau}_{\kappa(t)}\|_{L^\infty}^{12}+\|u^{N,\tau}_{\kappa(t)}\|_{\dot{H}^{\beta}}^{12} \right) \\
			&\leq \epsilon \|\mathcal{E}_t\|^2 + C(\epsilon) \tau^{2\beta} \left( \|u^{N,\tau}_{\kappa(t)}\|_{L^\infty}^{18}+\|u^{N,\tau}_{\kappa(t)}\|_{\dot{H}^{\beta}}^{18}+1 \right). 
		\end{aligned}
	\end{equation}
	Next we focus on the estimate of $K_2$. By the second-order Taylor expansion, we obtain 
	\begin{align*}
		F(u^{N,\tau}_{t}) - F(u^{N,\tau}_{\kappa(t)}) 
		&= F'(u^{N,\tau}_{\kappa(t)}) (u^{N,\tau}_{t}-u^{N,\tau}_{\kappa(t)}) \\
		&\quad + \underbrace{ \int_{0}^{1} F''(u^{N,\tau}_{\kappa(t)}+\lambda(u^{N,\tau}_t-u^{N,\tau}_{\kappa(t)})) ( u^{N,\tau}_t-u^{N,\tau}_{\kappa(t)},u^{N,\tau}_t-u^{N,\tau}_{\kappa(t)} ) (1-\lambda) d\lambda }_{\text{denote it by } R_F}. 
	\end{align*}
	Since 
	\begin{align*}
		u^{N,\tau}_t-u^{N,\tau}_{\kappa(t)} = (S_N(t-\kappa(t))-I) u^{N,\tau}_{\kappa(t)} + \int_{\kappa(t)}^{t} G(s)S_N(t-s)F_N(u^{N,\tau}_{\kappa(s)}) ds
		+ \int_{\kappa(t)}^{t} S_N(t-s) dW(s), 
	\end{align*}
	we know that 
	\begin{align*}
		&K_2 = \left\langle \mathcal{E}_t , F'(u^{N,\tau}_{\kappa(t)}) (u^{N,\tau}_{t}-u^{N,\tau}_{\kappa(t)}) + R_F \right\rangle \\
		&= \left\langle \mathcal{E}_t , F'(u^{N,\tau}_{\kappa(t)}) (S_N(t-\kappa(t))-I) u^{N,\tau}_{\kappa(t)} \right\rangle 
		+ \left\langle \mathcal{E}_t , F'(u^{N,\tau}_{\kappa(t)}) \int_{\kappa(t)}^{t} G(s)S_N(t-s)F_N(u^{N,\tau}_{\kappa(s)}) ds \right\rangle \\
		&\quad + \left\langle \mathcal{E}_t , F'(u^{N,\tau}_{\kappa(t)}) \int_{\kappa(t)}^{t} S_N(t-s) dW(s) \right\rangle 
		+ \langle \mathcal{E}_t , R_F \rangle \\
		&=: K_{21} + K_{22} + K_{23} + K_{24}. 
	\end{align*}
	For $K_{21}$, by Young's inequality, \eqref{Est:negative norm} with $\theta=\beta$, H\"older's inequality and \eqref{semigroup2}, we have 
	\begin{equation}\label{proof K21}
		\begin{aligned}
			&\mathbb{E}[K_{21}] = \mathbb{E} \left[ \left\langle A^{\frac{1}{2}} \mathcal{E}_t , A^{-\frac{1}{2}} F'(u^{N,\tau}_{\kappa(t)}) (S_N(t-\kappa(t))-I) u^{N,\tau}_{\kappa(t)} \right\rangle \right] \\
			&\leq \mathbb{E} \left[ \epsilon \|\mathcal{E}_t\|_{\dot{H}^1}^2 + \frac{C}{4\epsilon} \left( 1+\max\{ \|u^{N,\tau}_{\kappa(t)}\|_{L^\infty}^4 , \|u^{N,\tau}_{\kappa(t)}\|_{\dot{H}^{\beta}}^4\} \right)  \|(S(t-\kappa(t))-I)u^{N,\tau}_{\kappa(t)}\|_{\dot{H}^{-\beta}}^2 \right] \\
			&\leq \epsilon \mathbb{E}[\|\mathcal{E}_t\|_{\dot{H}^1}^2] + \frac{C}{4\epsilon} \left( \mathbb{E} \left[ 1+ \|u^{N,\tau}_{\kappa(t)}\|_{L^\infty}^8 + \|u^{N,\tau}_{\kappa(t)}\|_{\dot{H}^{\beta}}^8 \right] \right)^{\frac{1}{2}}   
			\tau^{2\beta} \left( \mathbb{E}[\|u^{N,\tau}_{\kappa(t)}\|_{\dot{H}^{\beta}}^4] \right)^{\frac{1}{2}} \\
			&\leq \epsilon \mathbb{E}[\|\mathcal{E}_t\|_{\dot{H}^1}^2] + \frac{C}{4\epsilon} \tau^{2\beta} \left( 1 +  \sup\limits_{t\geq0}\|u^{N,\tau}_{t}\|_{L^8(\Omega;L^\infty)}^8 + \sup\limits_{t\geq0}\|u^{N,\tau}_{t}\|_{L^8(\Omega;\dot{H}^{\beta})}^8 \right). 
		\end{aligned}
	\end{equation}
	For $K_{22}$, through Young's inequality, \eqref{Est:F2} and H\"older's inequality, we obtain 
	\begin{equation}\label{proof K22}
		\begin{aligned}
			&\mathbb{E}[K_{22}] \leq \epsilon \mathbb{E}[\|\mathcal{E}_t\|^2] + \frac{1}{4\epsilon} \mathbb{E} \left[\left( \int_{\kappa(t)}^{t} \left\|F'(u^{N,\tau}_{\kappa(t)}) S_N(t-s)F_N(u^{N,\tau}_{\kappa(s)}) \right\| ds \right)^2\right] \\
			&\leq \epsilon \mathbb{E}[\|\mathcal{E}_t\|^2] + \frac{C\tau}{4\epsilon} \mathbb{E}\left[ \int_{\kappa(t)}^{t} \left(1+\|u^{N,\tau}_{\kappa(t)}\|_{L^\infty}^4\right) \|F(u^{N,\tau}_{\kappa(s)})\|^2 ds \right] \\
			&\leq \epsilon \mathbb{E}[\|\mathcal{E}_t\|^2] + \frac{C\tau^2}{4\epsilon} \left( 1+\sup\limits_{t\geq0}\|u^{N,\tau}_t\|_{L^{10}(\Omega;L^\infty)}^{10} \right). 
		\end{aligned}
	\end{equation}
	For $K_{24}$, by Young's inequality, \eqref{Est:F2}, H\"older's inequality, \eqref{regularity v} and \eqref{Holder v}, we have 
	\begin{equation}\label{proof K24}
		\begin{aligned}
			&\mathbb{E}[K_{24}] \\ 
			=& \  \mathbb{E}\left[\left\langle A^{\frac{1}{2}} \mathcal{E}_t , A^{-\frac{1}{2}} \int_{0}^{1} F''(u^{N,\tau}_{\kappa(t)}+\lambda(u^{N,\tau}_t-u^{N,\tau}_{\kappa(t)})) \big( u^{N,\tau}_t-u^{N,\tau}_{\kappa(t)} , u^{N,\tau}_t-u^{N,\tau}_{\kappa(t)} \big) (1-\lambda) d\lambda \right\rangle\right] \\
			\leq & \  \epsilon \mathbb{E}[\|\mathcal{E}_t\|_{\dot{H}^1}^2] + \frac{C_F}{4\epsilon} \mathbb{E} \left[\left( \|u^{N,\tau}_{\kappa(t)}\|_{L^\infty}^2 + \|u^{N,\tau}_t\|_{L^\infty}^2 + 1 \right)  \|u^{N,\tau}_t-u^{N,\tau}_{\kappa(t)}\|^4\right]  \\
			\leq & \  \epsilon \mathbb{E}[\|\mathcal{E}_t\|_{\dot{H}^1}^2] + \frac{C}{4\epsilon} \tau^{2\beta}. 
		\end{aligned}
	\end{equation}
	Next we estimate $K_{23}$. Obviously, $\mathcal{E}_{\kappa(t)}$ and $F'(u^{N,\tau}_{\kappa(t)})$ are $\mathcal{F}_{\kappa(t)}$ measurable, which means that 
	\begin{align*}
		&\mathbb{E}[K_{23}] = \mathbb{E} \Big[\big\langle \mathcal{E}_t , F'(u^{N,\tau}_{\kappa(t)}) \int_{\kappa(t)}^{t} S_N(t-s) dW(s) \big\rangle\Big] \\
		&= \mathbb{E} \Big[\big\langle \mathcal{E}_t - \mathcal{E}_{\kappa(t)} , F'(u^{N,\tau}_{\kappa(t)}) \int_{\kappa(t)}^{t} S_N(t-s) dW(s) \big\rangle\Big]
		+ \underbrace{\mathbb{E} \Big[\big\langle \mathcal{E}_{\kappa(t)} , F'(u^{N,\tau}_{\kappa(t)}) \int_{\kappa(t)}^{t} S_N(t-s) dW(s) \big\rangle \Big]}_{=0}. 
	\end{align*}
	Denote 
	\begin{align*}
		M_{\kappa(t)} :=& \ (S_N(t-\kappa(t))-I) \mathcal{E}_{\kappa(t)}-\int_{\kappa(t)}^{t}  G(s)S_N(t-s)F_N(u^{N,\tau}_{\kappa(s)}) ds \\ 
		=& \ (S_N(t-\kappa(t))-I) \mathcal{E}_{\kappa(t)} 
		- A_N^{-1} (I-S_N(t-\kappa(t)))  G(\kappa(t)) F_N(u^{N,\tau}_{\kappa(t)}), 
	\end{align*}
	then $M_{\kappa(t)}$ is $\mathcal{F}_{\kappa(t)}$ measurable, which means that  
	\begin{align*}
		\mathbb{E} \left[\left\langle M_{\kappa(t)} , F'(u^{N,\tau}_{\kappa(t)}) \int_{\kappa(t)}^{t} S_N(t-s) dW(s) \right\rangle\right] = 0. 
	\end{align*}
	Since $\int_{\kappa(t)}^{t}  S_N(t-s)F_N(u^N(\kappa(t))) ds$ is $\mathcal{F}_{\kappa(t)}$ measurable, we obtain 
	\begin{align*}
		\mathbb{E} \left[\left\langle \int_{\kappa(t)}^{t}  S_N(t-s)F_N(u^N(\kappa(t))) ds , F'(u^{N,\tau}_{\kappa(t)}) \int_{\kappa(t)}^{t} S_N(t-s) dW(s) \right\rangle\right] = 0. 
	\end{align*}
	From
	\begin{align*}
		\mathcal{E}_t-\mathcal{E}_{\kappa(t)} 
		&= (S_N(t-\kappa(t))-I) \mathcal{E}_{\kappa(t)} + \int_{\kappa(t)}^{t} \left( S_N(t-s)F_N(u^N(s)) -  G(s)S_N(t-s)F_N(u^{N,\tau}_{\kappa(s)}) \right) ds \\
		&= M_{\kappa(t)} + \int_{\kappa(t)}^{t} S_N(t-s)F_N(u^N(s)) ds,  
	\end{align*}
	and H\"older's inequality, \eqref{Est:F2}, \eqref{Holder F(uN)}, \eqref{regularity v} and \eqref{Holder O} with $\gamma=0$, $p=4$, we have 
	\begin{equation}\label{proof:K23}
		\begin{aligned}
			&\mathbb{E}[K_{23}] \\
			= \ &  \mathbb{E} \Big[\Big\langle M_{\kappa(t)} + \int_{\kappa(t)}^{t} S_N(t-s)F_N(u^N(s)) ds , F'(u^{N,\tau}_{\kappa(t)}) \int_{\kappa(t)}^{t} S_N(t-s) dW(s) \Big\rangle\Big] \\
			= \ &  \mathbb{E} \Big[\Big\langle \int_{\kappa(t)}^{t} S_N(t-s) \left( F_N(u^N(s)) - F_N(u^N(\kappa(t))) \right)ds , F'(u^{N,\tau}_{\kappa(t)}) \int_{\kappa(t)}^{t} S_N(t-s) dW(s) \Big\rangle\Big] \\
			\leq \ & C \int_{\kappa(t)}^{t} \|F(u^N(s))-F(u^N(\kappa(t)))\|_{L^2(\Omega;H)} ds \\
			& \times \left(\mathbb{E} \Big[ \big(1+\|u^{N,\tau}_{\kappa(t)}\|_{L^\infty}^2\big)^2 \Big\| \int_{\kappa(t)}^{t} S_N(t-s) dW(s) \Big\|^2 \Big] \right)^{\frac{1}{2}} \\
			\leq \ & C \tau^{1+\frac{\beta}{2}} \left( 1+\sup\limits_{t\geq0} \|u^{N,\tau}_t\|_{L^8(\Omega;L^\infty)}^2 \right) 
			\Big\|  \int_{\kappa(t)}^{t} S_N(t-s) dW(s) \Big\|_{L^4(\Omega;H)}
			\leq C \tau^{1+\beta}. 
		\end{aligned}
	\end{equation}
	Finally, putting the estimates \eqref{proof K1}, \eqref{proof K3}, \eqref{proof K21}--\eqref{proof:K23} together, then by Poincar\'e's inequality, we obtain 
	\begin{align*}
		\frac{1}{2} \frac{d}{dt} \mathbb{E}[\|\mathcal{E}_t\|^2] &\leq -\mathbb{E}[\|\mathcal{E}_t\|_{\dot{H}^1}^2] + 4\epsilon \mathbb{E}[\|\mathcal{E}_t\|_{\dot{H}^1}^2] + L_F \mathbb{E}[\|\mathcal{E}_t\|^2] + C(\epsilon) (\tau^{2\beta}+\tau^2+\tau^{1+\beta}) \\
		&\leq -(\lambda_{1}-4\epsilon\lambda_{1}-L_F) \mathbb{E}[\|\mathcal{E}_t\|^2] + C(\epsilon,\tau_0) \tau^{2\beta}. 
	\end{align*}
	By taking $\epsilon=\frac{\lambda_{1}-L_F}{8\lambda_{1}}$ and noting that $\lambda_{1}>L_F$ and $\mathcal{E}_0=0$, it follows from Gronwall's inequality that 
	\begin{align*}
		\mathbb{E}[\|\mathcal{E}_t\|^2] 
		\leq C \tau^{2\beta} \int_{0}^{t} e^{-(\lambda_{1}-L_F)(t-s)} ds
		\leq C\tau^{2\beta}, 
	\end{align*}
	where the constant $C$ is independent of $N$, $\tau$ and $t$. This means that 
	\begin{align*}
		\left| \mathbb{E}[\Phi(u^N(T))] - \mathbb{E}[\Phi(u^{N,\tau}_T)] \right|
		\leq C \|\mathcal{E}_T\|_{L^2(\Omega;H)}
		\leq C\tau^{\beta}. 
	\end{align*}
	Since the above estimate holds for any $T>0$, taking $T=K\tau$ and recalling $u^{N,\tau}_T=u^{N}_{K}$ yields the desired conclusion. 	
\end{proof}

From \eqref{exponential decay} with $p=2$, Remark \ref{remark} and Theorem \ref{Th:full-discrete order}, we have 
\begin{align*}
	&\left| \mathbb{E}[\Phi(u^{N}_K)] - \int_{H} \Phi d\mu \right| 
	\leq \left| \mathbb{E}[\Phi(u(K\tau))] - \int_{H} \Phi d\mu \right| 
	+ \left| \mathbb{E}[\Phi(u(K\tau))] - \mathbb{E}[\Phi(u^{N}_K)] \right| \\ 
	&\leq  C(\Phi) e^{-(\lambda_{1}-L_F)K\tau} \left(1 + \mathbb{E}[\|u_0\|]\right) 
	+ C (1+(K\tau)^{-\beta}) ( \lambda_{N+1}^{-\beta}+\tau^{\beta} ). 
\end{align*}
Taking $K$ sufficiently large yields the conclusion of Corollary \ref{Corollary:1}. 

\begin{Remark}
	We note that Corollary \ref{Corollary 2} can ensure the existence of the invariant measure of $u^{N,\tau}_t$ through the Krylov--Bogoliubov theorem. 
	If $\mu^{N,\tau}$ is such an ergodic invariant
	measure of $u^{N,\tau}_t$, then by taking $T$ sufficiently large in Corollary \ref{Corollary:1}, we get the weak convergence between $\mu^{N,\tau}$ and $\mu$, 
	\begin{align*}
		\left| \int_{H} \Phi d\mu^{N,\tau} - \int_{H} \Phi d\mu \right|
		\leq C(u_0,Q,\Phi) (\lambda_{N+1}^{-\beta}+\tau^{\beta}). 
	\end{align*}
	However, whether invariant measure of our numerical scheme is unique, is still unknown, which is our next effort. 
\end{Remark}

\section{Numerical Experiment}\label{Sec:numerical example}
In this section we present a numerical example to validate the theoretical results. We consider the following SACE driven by a $Q$-Wiener process
\begin{equation*}
	\left\{
	\begin{aligned}
		&\frac{\partial u(t, x)}{\partial t} = \frac{\partial^2 u(t, x)}{\partial x^2} + u(t, x) - u^3(t, x) + \dot{W}(t), \quad t \in (0, T], \; x \in (0, 1),\\[2mm]
		&u(t, 0) = u(t, 1) = 0, \quad t \in [0, T],\\[2mm]
		&u(0, x) = \sin(\pi x), \quad x \in (0, 1).
	\end{aligned}
	\right.
\end{equation*}  

We first choose the test function $\Phi(x) = \sin(\|x\|)$ to examine the weak convergence rate. The mathematical expectation is approximated by averaging over 20000 independent sample paths, yielding the weak error at the terminal time $T = 1$. Since the analytical solution is unavailable, the numerical solution with $N = 1000$ and $\tau = 2^{-15}$ is taken as the reference solution. Table \ref{table:error} presents temporal weak errors and weak convergence rates for different time step sizes \(\tau = 2^{-i}, \; i = 4,5,\ldots,8\). From the numerical results we observe that for the case of space-time white noise (i.e., $\beta = 0.5 - \epsilon$), the temporal weak convergence rate is approximately $0.5$, while for the case of trace-class noise (i.e., $\beta = 1$), the temporal weak convergence rate is approximately $1$, which agrees with the theoretical result (see Theorem \ref{Th:full-discrete order}).

\begin{table}[htbp]
	\centering
	\caption{\centering Temporal weak errors and convergence rates.}
	\label{table:error}
	\renewcommand\arraystretch{1.2}
	\begin{tabular}{c|cc|cc}
		\hline
		$\tau$ & \multicolumn{2}{c|}{space-time white noise ($\beta=0.5-\epsilon$)} & \multicolumn{2}{c}{trace-class noise ($\beta=1$)} \\ \hhline{~----}
		& error & rate & error & rate \\
		\hline
		$2^{-4}$ & 3.2282e-03 &          & 2.0902e-03 &          \\
		$2^{-5}$ & 2.4684e-03 & 0.3872  & 1.2478e-03 & 0.7442  \\
		$2^{-6}$ & 1.7870e-03 & 0.4660  & 6.8296e-04 & 0.8695  \\
		$2^{-7}$ & 1.2656e-03 & 0.4978  & 3.5327e-04 & 0.9510  \\
		$2^{-8}$ & 8.9423e-04 & 0.5011  & 1.8373e-04 & 0.9432  \\
		\hline
		expected rate &  & $0.5-\epsilon$ &  & $1$ \\
		\hline
	\end{tabular}
\end{table}

We next investigate how the weak errors evolve over time. 
We take $T=50$ and $N=1000$, and compute two numerical solutions: one with a fine time step size $\tau_1=T/2^{14}$ as the reference solution, and another with a coarser step size $\tau_2=T/2^{10}$. For different test functions $\Phi$, the weak error $|\mathbb{E}[\Phi(u^{N,\tau_1}_k)]-\mathbb{E}[\Phi(u^{N,\tau_2}_k)]|$ is calculated up to $T=50$. 
As shown in Figure \ref{figure1}, the weak errors for all test functions remain bounded and are non-increasing in time, which confirms the suitability of our tamed scheme for long-time simulations.

Finally, we consider the long-time behaviors of the numerical solution. Through Corollary \ref{Corollary:1} we know that $\mathbb{E}[\Phi(u^N_k)]$ with different initial values will converge to $\int_{H} \Phi d\mu$. To verify this property, we simulate $\mathbb{E}[\Phi(u^N_k)]$ started from different initial values with the terminal time $T=10$, as shown in Figure \ref{figure2}. It can be observed that for three different test functions $\Phi$, the values of $\mathbb{E}[\Phi(u^N_k)]$ simulated from different initial values converge to the same value in a short time, which is consistent with the result of Corollary \ref{Corollary:1}.

\begin{figure}[!htp]
	\begin{center}
		\includegraphics[scale = 0.4]{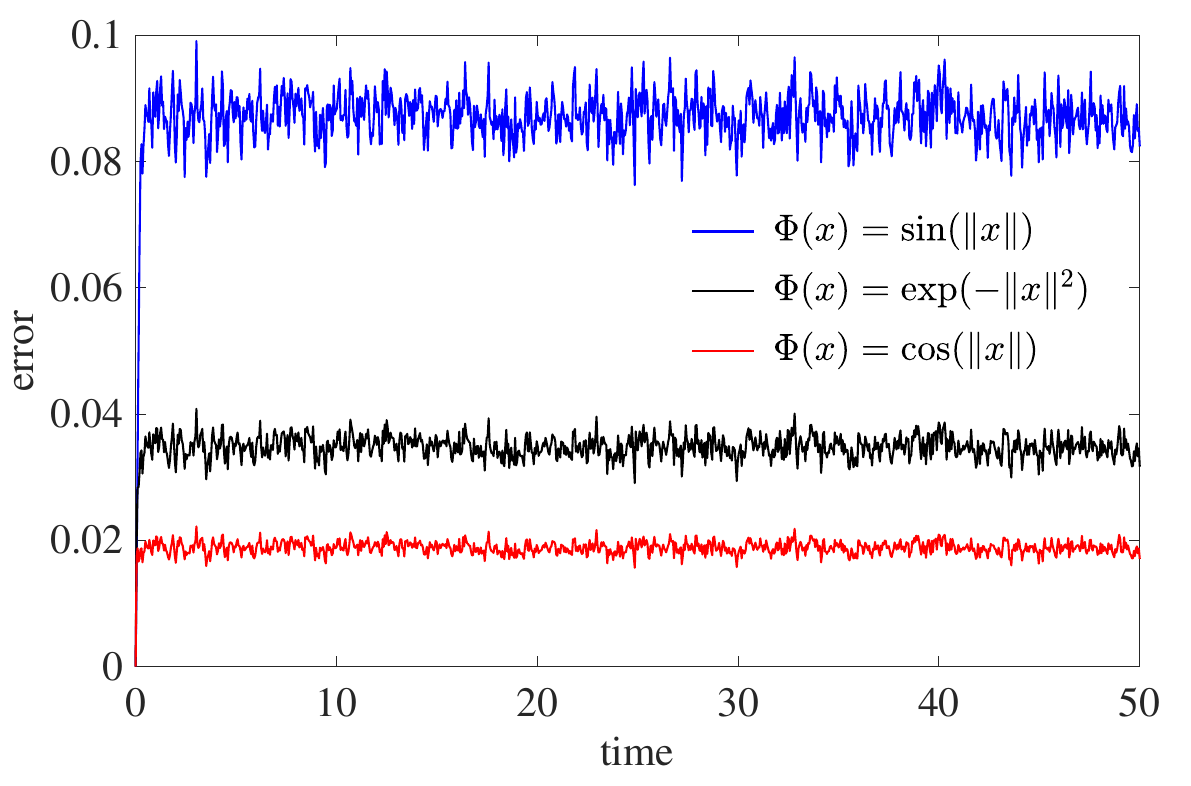}
		 \caption{\centering Evolution of weak errors for different test functions. }\label{figure1}
	\end{center}	
\end{figure}

\begin{figure}[!htp]
	\begin{center}
		\includegraphics[scale = 0.27]{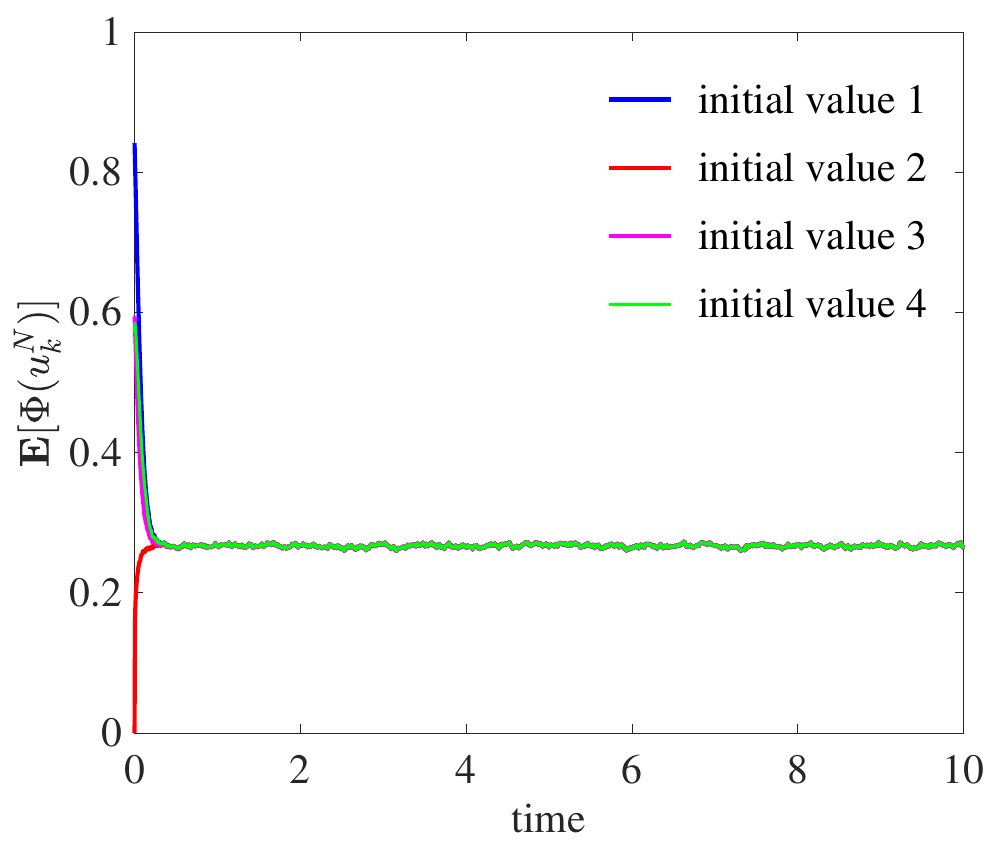}
		\includegraphics[scale = 0.27]{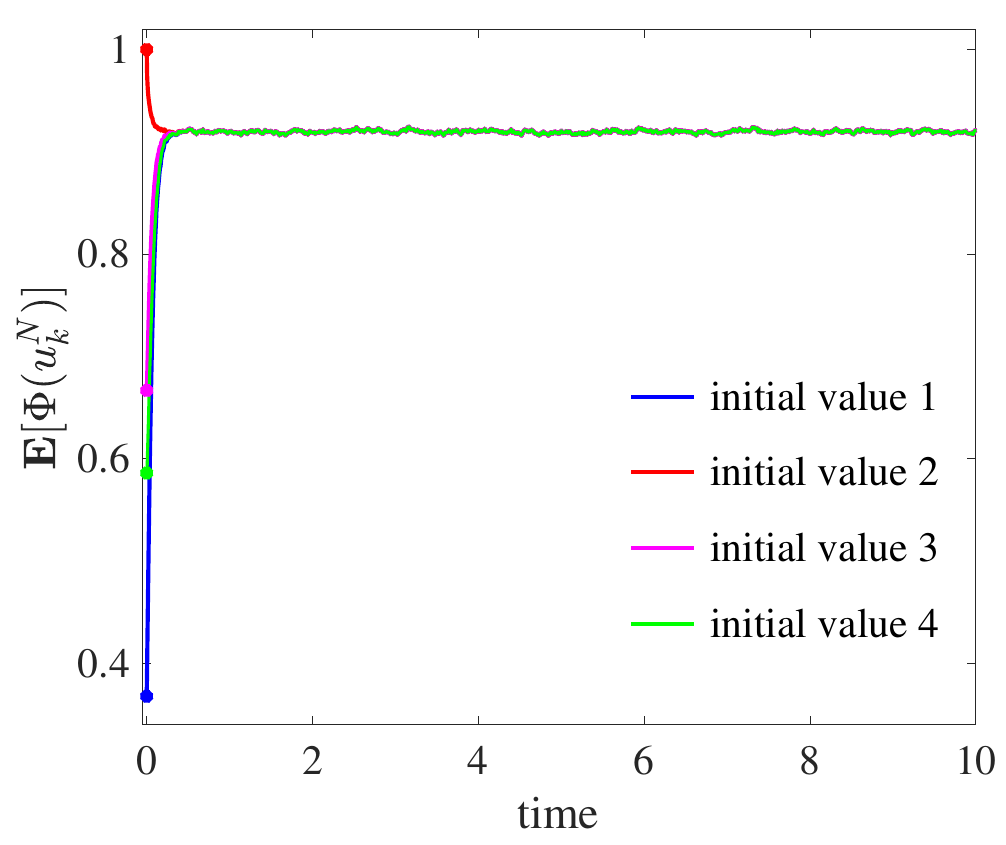}
		\includegraphics[scale = 0.27]{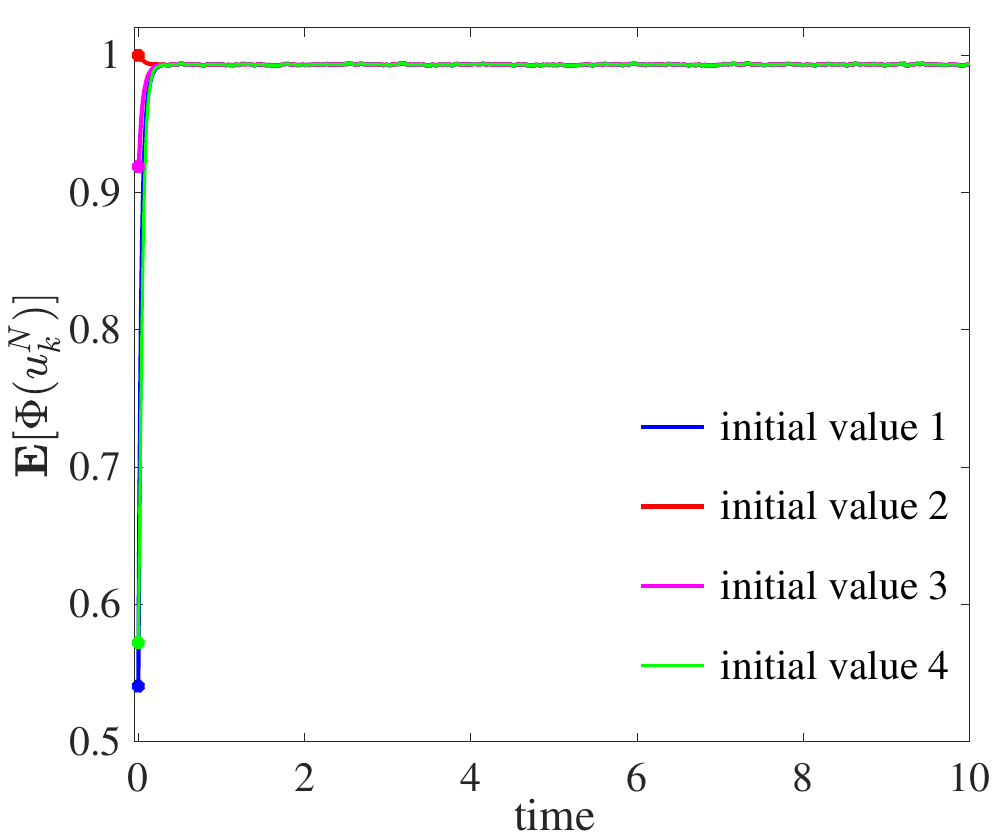}
	\caption{\centering Evolution of $\mathbb{E}[\Phi(u^N_k)]$ started from different initial values. Left: $\Phi(x)=\sin(\|x\|)$; Middle $\Phi(x)=\exp(-\|x\|^2)$; Right: $\Phi(x)=\cos(\|x\|)$ }\label{figure2}
	\end{center} 
\end{figure}

\section{Conclusion}\label{section:conclusion}
In this paper, we have proposed a novel explicit fully discrete scheme for the SACE, leveraging the spectral Galerkin method for spatial discretization and a tamed accelerated exponential integrator for temporal approximation. This approach addresses the computational challenges associated with implicit methods, enabling efficient long-time simulations and providing an explicit numerical approximation of the invariant measure. Through the use of time-independent moment bounds and Malliavin calculus, we have derived a uniform weak convergence analysis, demonstrating that the proposed scheme effectively approximates the invariant measure of the original problem.

The key innovation of this work lies in the introduction of a refined taming strategy, which guarantees the uniform-in-time moment boundedness of numerical solution (see Theorem \ref{Th:moment bound v}). To our knowledge, this is the first work to establish such a result for an explicit fully discrete method applied to the SACE. The weak error estimates provided in this paper further confirm the accuracy and efficiency of the proposed method, making it a promising tool for explicitly approximating invariant measures in SPDEs.

Future research will investigate the ergodicity of the numerical invariant measure and further quantify the approximation error between the numerical and exact invariant measures. Additionally, this study provides a foundation for extending these techniques to other nonlinear SPDEs, thereby contributing to a deeper understanding and more accurate numerical analysis of the long-term behavior of stochastic phase field models and related equations.

%
%
%
%

%

\begin{appendices}

\section{Proof of Lemma \ref{lem:lem3.2}}\label{Appendix perturbed} 

Obviously, $\|z\|_{\mathbb{L}_{a}^p} \leq \|z\|_{\mathbb{L}_{b}^p}$ holds for $b \leq a$. 
In addition, for $a \in (0,1)$ and $p \geq 2$, a straightforward calculation gives
\begin{equation}\label{Est:singular w LL}
	\begin{aligned}
		&\int_{0}^{t} (t-s)^{-a} e^{-\frac{\lambda_{1}(t-s)}{4}} \|z\|_{\mathbb{L}^p_{\lambda_{1}/4}(\mathcal{I}\times [0,s])}^p  ds
		\leq \int_{0}^{t} (t-s)^{-a} e^{-\frac{\lambda_{1}(t-s)}{4}} \|z\|_{\mathbb{L}^p_{\lambda_{1}/8}(\mathcal{I}\times [0,s])}^p  ds \\
		&\leq \left( \int_{0}^{t} (t-s)^{-a} e^{-\frac{\lambda_{1}(t-s)}{8}} ds \right) 
		\left( \int_{0}^{t} e^{-\frac{\lambda_{1}(t-r)}{8}} \|z(r)\|_{L^p}^p dr \right) \\
		&\leq  (\tfrac{8}{\lambda_{1}})^{1-a} \Gamma(1-a)  \|z\|_{\mathbb{L}^p_{\lambda_{1}/8}(\mathcal{I}\times [0,t])}^p.  
	\end{aligned}
\end{equation}
\begin{proof}[Proof of Lemma \ref{lem:lem3.2}.]

Since $w(0) = 0$, \eqref{Est:w L8} is obviously true when $t=0$. In the following we focus on the case of $t>0$. 
	Through the integration by parts, we have 
	\begin{equation}\label{perturbed 2}
			\frac{1}{2} \frac{d}{dt} \left(e^{\varrho t} \|w\|^2\right) 
			= \frac{1}{2} \left( e^{\varrho t} \frac{d}{dt} \|w\|^2 + \varrho e^{\varrho t} \|w\|^2 \right) 
			= e^{\varrho t} \left( -\|w\|_{\dot{H}^1}^2 + \langle w , F(w+z) \rangle + \frac{\varrho}{2} \|w\|^2 \right). 
	\end{equation}
	It is easy to compute that for $a_3>0$, 
	\begin{align*}
		\langle w , F(w+z) \rangle &\leq -a_3\|w\|_{L^4}^4 + 3a_3\|w\|_{L^4}^3\|z\|_{L^4} + a_3\|w\|_{L^4}\|z\|_{L^4}^3 + |a_2|\|w\|_{L^4}^3 + 2|a_2|\|w\|_{L^4}^2\|z\|_{L^4} \\
		&\quad + |a_2|\|w\|_{L^4}\|z\|_{L^4}^2 + |a_1|\|w\|_{L^4}^2 + |a_1|\|w\|_{L^4}\|z\|_{L^4} + |a_0|\|w\|_{L^4}. 
	\end{align*}
	Substituting the above estimate into \eqref{perturbed 2}, and denoting $\mathbb{G}_{t,\varrho} := \frac{1}{\varrho}(1-e^{-\varrho t})$, then by H\"older's inequality we get 
	\begin{equation}\label{perturbed 3}
		\begin{aligned}
			\tfrac{1}{2} \|w(t)\|^2 
			&\leq -a_3\|w\|_{\mathbb{L}_{\varrho}^4}^4  + 3a_3\|w\|_{\mathbb{L}_{\varrho}^4}^3\|z\|_{\mathbb{L}_{\varrho}^4} 
			+ a_3\|w\|_{\mathbb{L}_{\varrho}^4}\|z\|_{\mathbb{L}_{\varrho}^4}^3 
			+ |a_2|\|w\|_{\mathbb{L}_{\varrho}^4}^3 \mathbb{G}_{t,\varrho}^{\frac{1}{4}} \\ 
			& \quad + 2|a_2|\|w\|_{\mathbb{L}_{\varrho}^4}^2\|z\|_{\mathbb{L}_{\varrho}^4} \mathbb{G}_{t,\varrho}^{\frac{1}{4}}
			+ |a_2|\|w\|_{\mathbb{L}_{\varrho}^4}\|z\|_{\mathbb{L}_{\varrho}^4}^2 \mathbb{G}_{t,\varrho}^{\frac{1}{4}} 
			+ |a_1|\|w\|_{\mathbb{L}_{\varrho}^4}^2 \mathbb{G}_{t,\varrho}^{\frac{1}{2}} \\
			& \quad + |a_1|\|w\|_{\mathbb{L}_{\varrho}^4}\|z\|_{\mathbb{L}_{\varrho}^4} \mathbb{G}_{t,\varrho}^{\frac{1}{2}} 
			+ |a_0|\|w\|_{\mathbb{L}_{\varrho}^4} \mathbb{G}_{t,\varrho}^{\frac{3}{4}} 
			+ \tfrac{\varrho}{2} \|w\|_{\mathbb{L}_{\varrho}^4}^2 \mathbb{G}_{t,\varrho}^{\frac{1}{2}}. 
		\end{aligned}
	\end{equation}
	For any fixed $t>0$, we claim that 
	\begin{align*}
		\|w\|_{\mathbb{L}_{\varrho}^4(\mathcal{I}\times [0,t])} \leq 5 \|z\|_{\mathbb{L}_{\varrho}^4(\mathcal{I}\times [0,t])} \quad \text{or} \quad \|w\|_{\mathbb{L}_{\varrho}^4(\mathcal{I}\times [0,t])} \leq 5 \mathbb{G}_{t,\varrho}^{\frac{1}{4}} \max\big\{ |\tfrac{a_2}{a_3}| , |\tfrac{a_1}{a_3}|^{\frac{1}{2}} , |\tfrac{a_0}{a_3}|^{\frac{1}{3}} , |\tfrac{\varrho}{a_3}|^{\frac{1}{2}} \big\}. 
	\end{align*}
	If this claim is false, namely, there exists some $t>0$, such that 
	\begin{equation*}
		\|z\|_{\mathbb{L}_{\varrho}^4(\mathcal{I}\times [0,t])} 
		< \frac{1}{5} \|w\|_{\mathbb{L}_{\varrho}^4(\mathcal{I}\times [0,t])}
		\quad \text{and} \quad  \mathbb{G}_{t,\varrho}^{\frac{1}{4}} < \frac{1}{5} \|w\|_{\mathbb{L}_{\varrho}^4(\mathcal{I}\times [0,t])} 
		\min\big\{ |\tfrac{a_3}{a_2}| , |\tfrac{a_3}{a_1}|^{\frac{1}{2}} , |\tfrac{a_3}{a_0}|^{\frac{1}{3}} , |\tfrac{a_3}{\varrho}|^{\frac{1}{2}} \big\}, 
	\end{equation*}
	then by noting that $a_3>0$, it can be deduced from \eqref{perturbed 3} that
	\begin{align*}
		0 \leq \tfrac{1}{2} \|w(t)\|^2 
		&\leq \left(-a_3+\tfrac{3a_3}{5}+\tfrac{a_3}{125} 
		+ \tfrac{a_3}{5} + \tfrac{2a_3}{25} + \tfrac{a_3}{125} 
		+ \tfrac{a_3}{25} + \tfrac{a_3}{125} + \tfrac{a_3}{50} 
		+ \tfrac{a_3}{125} \right)
		\|w\|_{\mathbb{L}_{\varrho}^4}^4 \\
		&= -\tfrac{7a_3}{250} \|w\|_{\mathbb{L}_{\varrho}^4}^4 < 0, 
	\end{align*}
	which leads a contradiction. Thus it holds that for any $t\geq0$, 
	\begin{equation}\label{Est:w LL4}
		\|w\|_{\mathbb{L}_{\varrho}^4(\mathcal{I}\times [0,t])} 
		\leq 5 \|z\|_{\mathbb{L}_{\varrho}^4(\mathcal{I}\times [0,t])} + 5 \mathbb{G}_{t,\varrho}^{\frac{1}{4}} \max\big\{ |\tfrac{a_2}{a_3}| , |\tfrac{a_1}{a_3}|^{\frac{1}{2}} , |\tfrac{a_0}{a_3}|^{\frac{1}{3}} , |\tfrac{\varrho}{a_3}|^{\frac{1}{2}} \big\}
		\leq C \big( \|z\|_{\mathbb{L}_{\varrho}^4(\mathcal{I}\times [0,t])} + 1 \big), 
	\end{equation}
	where the constant $C$ depends on $\varrho$ and $a_i$, $i=0, 1, 2, 3$, but is independent of $N$ and $t$. 
	
	From \eqref{perturbed 1}, by \eqref{Est:PN S4}, H\"older's inequality and \eqref{Est:w LL4}, we obtain  
	\begin{equation}\label{Est:w L3}
		\begin{aligned}
			&\|w(t)\|_{L^3} 
			\leq C \int_{0}^{t} (t-s)^{-\frac{5}{24}} e^{-\frac{\lambda_{1}(t-s)}{4}} \| F(w(s) + z(s)) \|_{L^{\frac{4}{3}}} ds \\ 
			&\leq C \left(\int_{0}^{t} (t-s)^{-\frac{5}{6}} e^{-\frac{\lambda_{1}(t-s)}{4}} ds \right)^{\frac{1}{4}} 
			\left(\int_{0}^{t} e^{-\frac{\lambda_{1}(t-s)}{4}} \left( \|w(s)\|_{L^{4}}^4 + \|z(s)\|_{L^{4}}^4 + 1 \right) ds\right)^{\frac{3}{4}} \\ 
			&\leq C \big( \|w\|_{\mathbb{L}_{\lambda_{1}/4}^4(\mathcal{I}\times [0,t])}^3 + \|z\|_{\mathbb{L}_{\lambda_{1}/4}^4(\mathcal{I}\times [0,t])}^3 + 1 \big)
			\leq C \big( \|z\|_{\mathbb{L}_{\lambda_{1}/4}^4(\mathcal{I}\times [0,t])}^3 + 1 \big). 
		\end{aligned}
	\end{equation}
	Finally, by \eqref{Est:PN S3}, \eqref{Est:w L3}, H\"older's inequality, \eqref{Est:singular w LL} with $a=\frac{1}{2}$ and $p=9$, and  $\|z\|_{\mathbb{L}_{\varrho}^q} \leq \|z\|_{\mathbb{L}_{\varrho}^p}$ for $p \geq q$, we have 
	\begin{align*}
		&\|w(t)\|_{L^\infty} 
		\leq C \int_{0}^{t} (t-s)^{-\frac{1}{2}} e^{-\frac{\lambda_{1}(t-s)}{4}} \left( \|w(s)\|_{L^3}^3 + \|z(s)\|_{L^3}^3 + 1 \right) ds \\
		&\leq C (\tfrac{4}{\lambda_{1}})^{\frac{1}{2}} \Gamma(\tfrac{1}{2})  
		+ C \int_{0}^{t} (t-s)^{-\frac{1}{2}} e^{-\frac{\lambda_{1}(t-s)}{4}} \|z\|_{\mathbb{L}_{\lambda_{1}/4}^4(\mathcal{I}\times [0,s])}^9 ds \\
		&\quad + C \left(\int_{0}^{t} (t-s)^{-\frac{3}{4}} e^{-\frac{\lambda_{1}(t-s)}{4}} ds\right)^{\frac{2}{3}} \left(\int_{0}^{t}  e^{-\frac{\lambda_{1}(t-s)}{4}} \|z(s)\|_{L^3}^9  ds\right)^{\frac{1}{3}} \\
		&\leq C \big(  \|z\|_{\mathbb{L}_{\lambda_{1}/8}^9(\mathcal{I}\times [0,t])}^9 + 1 \big), 
	\end{align*}
	which completes the proof. 
\end{proof}

\end{appendices}

\bibliographystyle{plain}

\begin{thebibliography}{99}
	
	\bibitem{Bally1996} V. Bally, D. Talay, The law of the Euler scheme for stochastic differential equations: I. Convergence rate of the distribution function, Probab. Theory Relat. Fields, 104 (1996) 43--60. 
	
	\bibitem{Bossy2021} M. Bossy, J.F. Jabir, K. Mart\'inez, On the weak convergence rate of an exponential Euler scheme for SDEs governed by coefficients with superlinear growth, Bernoulli, 27 (2021) 312--347. 
	
	
	
	
	
	
	
	
	\bibitem{Brehier2014} C.-E. Br\'ehier, Approximation of the invariant measure with an Euler scheme for stochastic PDEs driven by space-time white noise, Potential Anal., 40 (2014) 1--40. 
	
	\bibitem{Brehier2016} C.-E. Br\'ehier, G. Vilmart, High order integrator for sampling the invariant distribution of a class of parabolic stochastic PDEs with additive space-time noise, SIAM J. Sci. Comput., 38 (2016) A2283--A2306. 
	
	\bibitem{Brehier2022} C.-E. Br\'ehier, Approximation of the invariant distribution for a class of ergodic SPDEs using an explicit tamed exponential Euler scheme, ESAIM Math. Model. Numer. Anal., 56 (2022) 151--175. 
	
	
	
	\bibitem{Brehier2019}  C.-E. Br\'ehier, J.B. Cui, J.L. Hong, Strong convergence rates of semidiscrete splitting approximations for the stochastic Allen--Cahn equation, IMA J. Numer. Anal., 39 (2019) 2096--2134.
	
	\bibitem{BrehierGoude2019} C.-E. Br\'ehier, L. Gouden\`ege, Analysis of some splitting schemes for the stochastic Allen--Cahn equation, Discrete Contin. Dyn. Syst. Ser. B, 24 (2019), 4169--4190. 
	
	\bibitem{Brehier2017} C.-E. Br\'ehier, M. Kopec, Approximation of the invariant law of SPDEs: error analysis using a Poisson equation for a full-discretization scheme, IMA J. Numer. Anal., 37 (2017) 1375--1410. 
	
	
	\bibitem{Cai2021} M. Cai, S.Q. Gan, X.J. Wang, Weak convergence rates for an explicit full-discretization of stochastic Allen--Cahn equation with additive noise, J. Sci. Comput., 86 (2021) 34. 
	
	\bibitem{Cerrai2001} S. Cerrai, Second order PDE's in finite and infinite dimension: a probabilistic approach, Lecture Notes in Math., Springer, 2001.
	
	\bibitem{Chen2023} C.C. Chen, T.H. Dang, J.L. Hong, T. Zhou, CLT for approximating ergodic limit of SPDEs via a full discretization, Stochastic Process. Appl., 157 (2023) 1--41. 
	
	\bibitem{ChenDangHong2024} C.C. Chen, T.H. Dang, J.L. Hong, Strong convergence of adaptive time-stepping schemes for the stochastic Allen--Cahn equation, IMA J. Numer. Anal., 45 (2025) 404--450. 
	
	
	\bibitem{Chen2020} Z.H. Chen, S.Q. Gan, X.J. Wang, A full-discrete exponential Euler approximation of the invariant measure for parabolic stochastic partial differential equations, Appl. Numer. Math., 157 (2020) 135--158. 
	
	\bibitem{Cui2021} J.B. Cui, J.L. Hong, L.Y. Sun, Weak convergence and invariant measure of a full discretization for parabolic SPDEs with non-globally Lipschitz coefficients, Stochastic Process. Appl., 134 (2021) 55--93. 
	
	\bibitem{Feng2017} X.B. Feng, Y.K. Li, Y. Zhang, Finite element methods for the stochastic Allen--Cahn equation with gradient-type multiplicative noise, SIAM J. Numer. Anal., 55 (2017) 194--216.
	
	\bibitem{HuangShen2023} C. Huang, J. Shen, Stability and convergence analysis of a fully discrete semi-implicit scheme for stochastic Allen--Cahn equations with multiplicative noise, Math. Comp., 92 (2023) 2685--2713. 
	
	
	
	
	
	
	
	\bibitem{Kruse2014-2} R. Kruse, Optimal error estimates of Galerkin finite element methods for stochastic partial differential equations with multiplicative noise, IMA J. Numer. Anal., 34 (2014) 217--251. 
	
	\bibitem{Kruse2014-1} R. Kruse, Strong and weak approximation of semilinear stochastic evolution equations, Lecture Notes in Math., Springer, 2014.
	
	
	\bibitem{Liu2024} Z.H. Liu, Numerical ergodicity of stochastic Allen--Cahn equation driven by multiplicative white noise, Commun. Math. Res., 41 (2025) 30--44. 
	
	\bibitem{Liu2023} Z.H. Liu, Numerical ergodicity and uniform estimate of monotone SPDEs driven by multiplicative noise, J. Comput. Math., 44 (2026) 84--102. 
	
	
	
	
	\bibitem{Nualart2006} D. Nualart, The Malliavin calculus and related topics, Springer, 2006. 
	
	\bibitem{QiWang2019} R.S. Qi,  X.J. Wang, Optimal error estimates of Galerkin finite element methods for stochastic Allen--Cahn equation with additive noise, J. Sci.  Comput., 80 (2019) 1171--1194. 
	
	\bibitem{XuCJ2023} X. Qi, Y.R. Zhang, C.J. Xu, An efficient approximation to the stochastic Allen--Cahn	equation with random diffusion coefficient field and multiplicative noise, Adv. Comput. Math., 49 (2023) 73. 
	
	
	
	\bibitem{Wang2020} X.J. Wang, An efficient explicit full-discrete scheme for strong approximation of stochastic Allen--Cahn equation, Stochastic Process. Appl., 130 (2020) 6271--6299. 
	
	
	
	\bibitem{WangZhaoZhang2024} X.J. Wang, Y.Y. Zhao, Z.Q. Zhang, Weak error analysis for strong approximation schemes of SDEs with super-linear coefficients, IMA J. Numer. Anal., 44 (2024) 3153--3185. 
	
	
	
	
	
	
	
\end{thebibliography}

\end{document}